\documentclass{amsart}
\usepackage[utf8]{inputenc}
\usepackage[pdftex]{graphicx}
\usepackage{amsfonts,amssymb,verbatim}    
\usepackage{amsthm}
\usepackage{amsmath}
\usepackage{latexsym}
\usepackage{fancyhdr}
\usepackage{nameref}
\usepackage{amsmath,amscd}
\usepackage[all,cmtip]{xy}
\usepackage[letter,center]{crop}
\usepackage{geometry}
\usepackage{multicol}
\usepackage{paralist}
\usepackage{todonotes}
\usepackage{diagrams}

\usepackage{tikz}
\usetikzlibrary{arrows.meta} 
\usetikzlibrary{matrix,arrows}

\usepackage[pagebackref,bookmarksopen]{hyperref}
\hypersetup{
	colorlinks=true,
	linkcolor=blue,
	citecolor=magenta,
	urlcolor=cyan,
}
\usepackage{bookmark} 

\newcommand{\Comment}[1]{}
\newcommand{\N}{\mathbb{N}}
\newcommand{\Z}{\mathbb{Z}}
\newcommand{\R}{\mathbb{R}}
\newcommand{\C}{\mathbb{C}}

\newcommand{\set}[2]{\ensuremath{\{\,{#1}\mid {#2}\,\}}}

\newcommand{\inpr}[3][]{\ensuremath{\langle#2,#3\rangle_{#1}}}
\newcommand{\norm}[1]{\Vert #1\Vert}
\newcommand{\abs}[1]{| #1|}

\DeclareMathOperator{\ind}{Ind}

\DeclareMathOperator{\im}{Im}
\DeclareMathOperator{\re}{Re}

\newcommand{\Ind}[1][]{\ind_{\mathrm{#1}}}
\newtheorem{theorem}{Theorem}
\newtheorem{corollary}[theorem]{Corollary}
\newtheorem{proposition}[theorem]{Proposition}
\newtheorem{lemma}[theorem]{Lemma}

\theoremstyle{definition}
\newtheorem{definition}[theorem]{Definition}

\newtheorem{remark}[theorem]{Remark}

\newtheorem{example}[theorem]{Example}

 \newtheorem{theorem*}{Theorem}
 \newtheorem{corollary*}[theorem*]{Corollary}
 \newtheorem{proposition*}[theorem*]{Proposition}

\title{Indexes of vector fields for mixed functions}
\author{Jos\'e Luis Cisneros-Molina}
\address{Instituto de Matem\'aticas, Unidad Cuernavaca\\ Universidad Nacional Aut\'onoma de M\'exico\\ Avenida Universidad s/n, 
Colonia Lomas de Chamilpa\\ CP62210, Cuernavaca, Morelos Mexico}
\email{jlcisneros@im.unam.mx}
\author{Agust\'in Romano-Vel\'azquez}
\address{Instituto de Matem\'aticas, Unidad Cuernavaca\\ Universidad Nacional Aut\'onoma de M\'exico\\ Avenida Universidad s/n,
Colonia Lomas de Chamilpa\\ CP62210, Cuernavaca, Morelos Mexico}
\email{agustin.romano@im.unam.mx}

\thanks{This research was supported by the projects CONACYT 253506 and UNAM-DGAPA-PAPIIT-IN105121.
The first author was supported by an UNAM-DGAPA-PASPA sabbatical scholarship.
The second author is partially supported by ERCEA 615655 NMST Consolidator Grant, CONACYT CB 2016-1 Num. 286447, FORDECYT 265667, TIFR Visiting Fellow, NKFIH Grants “Elvonal (Frontier)” KKP 126683, KKP 144148 and by Lend\"ulet
30001. The second author thanks Erd\H{o}s Center for its hospitality and for providing a perfect work environment during the semester Singularities and Low Dimensional Topology - 2023 Spring.}
\subjclass[2020]{Primary 14J17, 32C18; Secondary 57R45.}

\begin{document}
\begin{abstract}
A mixed function is a real analytic map $f\colon\C^n\to\C$ in the complex variables $z_1,\dots,z_n$ and their conjugates $\bar{z}_1,\dots,\bar{z}_n$.
In this article we define an integer valued index for vector fields $v$ with isolated singularity at $\mathbf{0}$ on real analytic varieties $V_f:=f^{-1}(0)$ defined by mixed functions $f$ with isolated critical point at 
$\mathbf{0}$. We call this index the \emph{mixed GSV-index} and it generalizes the classical GSV-index defined by Gomez-Mont, Seade and Verjovsky in \cite{Gomez-Mont-Seade-Verjovsky:IHFIS}, i.e., if the function $f$ is holomorphic, then the mixed GSV-index coincides with the GSV-index. Furthermore, the mixed GSV-index is a lifting to $\Z$ of the $\Z_2$-valued real GSV-index defined by Aguilar, Seade and Verjovsky in \cite{Aguilar-Seade-Verjovsky:IVFTIRAS}. As applications we prove that the mixed GSV-index is equal to the Poincar\'e-Hopf index of $v$ on a Milnor fiber. If $f$ also satisfies the strong Milnor condition, i.~e., for every $\epsilon>0$ (small enough) the map $\frac{f}{\|f\|}\colon \mathbb{S}_\epsilon \setminus L_f \to \mathbb{S}^1$ is a fiber bundle, we prove that the mixed GSV-index is equal to the curvatura integra of $f$ defined by Cisneros-Molina, Grulha and Seade in \cite{Cisneros-Grulha-Seade:OTRAM} based on the curvatura integra defined by Kervaire in \cite{M1956}.

\end{abstract}

\maketitle

\section*{Introduction}
Let $v$ be a continuous vector field with isolated singularities on a smooth closed manifold $M$, the classical Poincar\'e-Hopf theorem says that the Poincar\'e-Hopf index
of $v$ is the Euler-Poincar\'e characteristic $\chi(M)$ of $M$.
In the case of a singular variety, an index for vector fields was defined by G\'omez-Mont, Seade and Verjovsky~\cite{Gomez-Mont-Seade-Verjovsky:IHFIS}. This index is called the \emph{GSV-index}
and it is defined for vector fields on a complex hypersurface with an isolated singularity. Later, a GSV-index was defined for vector fields on a real complete intersection with an isolated singularity by Aguilar, Seade and Verjovsky~\cite{Aguilar-Seade-Verjovsky:IVFTIRAS}. This index is called the \emph{real GSV-index}. 
Both versions of the GSV-index are important invariants for singular varieties (see the book \cite{Brasselet-etal:VFSV}).
For example, Klehn~\cite{klehn_2005} relates the GSV-index of a vector field $v$ in a complex isolated complete intersection singularity to the dimension of some vector space. Using this approach Klehn gives a sufficient criterion for the existence of a particular deformation of $v$. Smirnov and Torres~\cite{zbMATH07174536} use the real GSV-index to characterize the existence of a $\text{Spin}(1,3)_0$-Lorentzian cobordism between two closed oriented $3$-manifolds. There are many other related articles such as \cite{Aleksandrov:IVFLDF,Brasselet-etal:ECRFJC,Callejas-Bedregal-etal:LGFFAS,Gomez-Mont:AFIVFHIS,Lehmann-etal:IHVFTSV,Seade:IVFCASS,Seade-Suwa:AFLCI,Seade-Suwa:RFIHF}. 

Let $f\colon \R^{2n} \to \R^2$  be a real analytic map with isolated critical point at the origin. Let $v$ be a continuous vector field on $V_f:=f^{-1}(0)$ with isolated singularity at the origin. Under these assumptions, one has two possibilities:
\begin{enumerate}
    \item If $f$ is holomorphic (under the natural identification $\R^2 \cong \C$), then the GSV-index is an integer.\label{it:holo}
    \item If $f$ is not holomorphic, then the real GSV-index is an integer modulo $2$.\label{it:real}
\end{enumerate}
This follows because the definition of the GSV-index uses the Stiefel manifold of complex frames in \eqref{it:holo} and the Stiefel manifold of real frames in \eqref{it:real}. Thus, there is a topological obstruction to define the real GSV-index as an integer for this type of maps. If the difference of the dimensions of the domain and the target is odd, then the real GSV-index is an integer (see \cite{Aguilar-Seade-Verjovsky:IVFTIRAS}).

A mixed function is a real analytic map $f\colon\C^n\to\C$ in the complex variables $z_1,\dots,z_n$ and their conjugates $\bar{z}_1,\dots,\bar{z}_n$. 
In this article we define an integer valued index for vector fields $v$ on real analytic varieties $V_f:=f^{-1}(0)$ defined by mixed functions called the \emph{mixed GSV-index}. 
The general idea of the construction of the mixed GSV-index is the following: let $\mathbf{z}\in V_f$ be a regular point and consider the vector $v(\mathbf{z})$. We just need to find a  vector field $v^{\perp}$ such that $(v^{\perp}(\mathbf{z}),v(\mathbf{z}))$ is a complex $2$-frame. By obstruction theory, it is easy to prove the existence of such vector field $v^{\perp}$. Furthermore, the election of $v^{\perp}$ can be made such that it is compatible with the real GSV index. Such families of vector fields are called \emph{frame homotopic} to $v$. 

The mixed GSV-index generalizes the classical GSV-index, i.e., if the function $f$ is holomorphic, then the mixed GSV-index coincides with the GSV-index. This is done in Corollary~\ref{cor:MGSVequalGSV}. Furthermore, the real GSV-index is the reduction modulo two of the mixed GSV-index. Hence, the mixed GSV-index is a lifting to $\Z$ of the real GSV-index. This is done in Theorem~\ref{theorem.soniguales}.

Recall that given a complex (real) analytic variety $V_f$ with an isolated singularity at the origin, its \emph{link} $L_f=V_f\cap \mathbb{S}^{2n-1}_\varepsilon$ is a smooth manifold for a sphere $\mathbb{S}^{2n-1}_\varepsilon$ centered at the origin of sufficiently small radious $\varepsilon$.
One has a \emph{Milnor's fibration} (see~\cite{Milnor:SPCH}) whose fiber $F$ can be regarded as a compact $(2n-2)$-manifold with boundary $\partial F=L_f$.
By the Transversal Isotopy Lemma, there is an ambient isotopy on $\mathbb{S}^{2n-1}_\varepsilon$ taking $L_f$ to $\partial F$ and $v$ can be regarded as a non-singular vector field on $\partial F$, which can be extended with isolated singularities to $F$.
It is known that (under some hypothesis) the (real) GSV-index coincide with the (modulo two) Poincar\'e-Hopf index in a Milnor fiber $F$ (see~\cite[\S3.2, \S4.3]{Brasselet-etal:VFSV}). In this article we generalize the previous statement proving that if $f$ is a mixed function, then the mixed GSV-index is equal to the Poincar\'e-Hopf index of $v$
on a Milnor fiber $F$ (Theorem~\ref{theorem.sonigualesPH}). 

The last application of the mixed GSV-index is related with the \emph{cuvatura integra} $\varphi(M)$ of a closed submanifold $M$ of the Euclidean space with trivial normal bundle defined by Kervaire~\cite{M1956}. Given a real analytic map $f\colon\R^n\to\R^k$ with isolated critical point and vector field $v$ on $V_f=f^{-1}(0)$ with isolated singularity at the origin, under some conditions the real GSV-index of $v$ coincides with the cuvatura integra $\varphi(L_f)$ of the link $L_f$ (see~\cite[Theorem~3.3]{Aguilar-Seade-Verjovsky:IVFTIRAS}).
Given a real analytic map $f\colon\R^n\to\R^k$ with isolated critical point such that the map $f/\Vert f\Vert\colon\mathbb{S}^{n-1}_\varepsilon\to\mathbb{S}^{k-1}$ is a fiber bundle, in \cite{Cisneros-Grulha-Seade:OTRAM} Cisneros-Molina, Grulha and Seade defined the curvatura integra $\Psi(f)$ of $f$ in terms of Kervaire's curvatura integra.
In the case that $f$ is a mixed function and $v$ is a vector field on $V_f$ with isolated singularity at the origin that is everywhere transverse to the link $L_f$,
we prove that the mixed GSV-index coincides with the curvatura integra $\Psi(f)$ of $f$ (Corollary~\ref{cor:CurvaturaIntegra}).

The organization of this paper is as follows: In \S~\ref{sec:GSV-Stiefel} we give preliminary results about the Stiefel manifolds, and the real and complex GSV-index. In \S~\ref{sec:mixed} we prove some properties about mixed functions that we will use later. In \S~\ref{sec:TheMGSV} we define the mixed GSV-index. In \S~\ref{Applications} we compare the GSV-index with the real GSV-index, the Poincar\'e-Hopf index and the curvartura integra.

\section*{Notation}
Throughout this article we consider $\C^n$ with coordinates $z_1, \dots, z_n$. Let $\bar{z}_j$ be the complex conjugate of $z_j$. We shall write $z_j = x_j +i y_j$ with $x_j,y_j \in \R$.
To simplify notation we shall write $\mathbf{z}=(z_1,\dots,z_n)$, $\bar{\mathbf{z}}=(\bar{z}_1,\dots,\bar{z}_n)$, 
$(\mathbf{x},\mathbf{y})=(x_1,y_1,\dots,x_n,y_n)$. Note the order of the variables in the latter.
We also denote by $\mathbf{0}$ the origin in $\C^n$ and by $\R^+$ the positive real numbers.  
To every vector $\mathbf{z}\in\C^n$ corresponds the vector $(\mathbf{x},\mathbf{y})\in\R^{2n}$ and \textit{vice versa}.
To simplify notation we shall write $\mathbf{z}_\R=(\mathbf{x},\mathbf{y})$ and $(\mathbf{x},\mathbf{y})_\C=\mathbf{z}$.

Endow $\C^n$ with the standard Hermitian inner product denoted by
$\inpr[\C^n]{-}{-}$.
Consider the standard Euclidean inner product on $\R^{2n}$ denoted by $\inpr[\R^{2n}]{-}{-}$. By the definitions, it is easy to check that
$\inpr[\R^{2n}]{-}{-}=\re\inpr[\C^n]{-}{-}$.

Let $f\colon U \subset\mathbb{R}^n\to\mathbb{R}^k$ be a real analytic function where $U$ is an open neighborhood of the origin in $\mathbb{R}^n$. From now on, we denote by
\begin{equation*}
 V_f:=f^{-1}(0), \quad  V_f^*:=V_f-\mathrm{Sing}\, V_f.
\end{equation*}
Let $\mathbb{S}^{n-1}_\varepsilon\subset\R^n$ be a sphere of  sufficiently small radius $\varepsilon>0$ with center at the origin. We denote the link of $f$ at $\mathbf{0}$ by
\begin{equation*}
    L_f:=V_f\cap\mathbb{S}^{2n-1}_\varepsilon.
\end{equation*} 

\section{GSV-index and real GSV-index}
\label{sec:GSV-Stiefel}
In this section we recall the definition of the GSV-index given in \cite{Gomez-Mont-Seade-Verjovsky:IHFIS} and of the real GSV-index given in \cite{Aguilar-Seade-Verjovsky:IVFTIRAS}. We begin by recalling some properties about Stiefel manifolds.

From now on, the \emph{Stiefel manifold} of real orthonormal $k$-frames in $\R^m$ is denoted by $V(k,m)$. Analogously, the Stiefel manifold of complex orthonormal $k$-frames in $\C^m$ is denoted by $W(k,m)$. Recall that  $V(k,m)$ is $m-k-1$-connected (see~\cite[\S25.6]{Steenrod:TFB})) and $W(k,m)$ is $2(m-k)$-connected (see~\cite[\S25.7]{Steenrod:TFB}), that is,
\begin{equation}\label{eq:pii.wkm}
\pi_i\bigl(W(k,m)\bigr)=\begin{cases}
                  0 & \text{for $i\leq 2(m-k)$,}\\
                  \Z & \text{for $i=2(m-k)+1$,}
                 \end{cases} \quad
\pi_i\bigl(V(k,m)\bigr)=\begin{cases}
                  0 & \text{for $i< m-k$,}\\
                  \Z & \text{if $i=m-k$ is even, or $k=1$,}\\
                  \Z_2 & \text{if $i=m-k$ is odd.}
                 \end{cases}
\end{equation}
The Stiefel manifold $V(k,m)$ fibers over $V(k-1,m)$ with fiber $S^{m-k}$ (see for instance \cite[\S2]{James:SM} or \cite[\S7]{Steenrod:TFB}), that is, we have the fiber bundle
\begin{equation}\label{eq:fib.gen.real}
\begin{gathered}
\xymatrix{\mathbb{S}^{m-k}\cong V(1,m-k+1)\ar@{^{(}->}[r]^(.7){\tilde{\gamma}} & V(k,m)\ar[r]^{\tilde{p}} & V(k-1,m),}\\[5pt]
\tilde{p}(v_1,v_2,\dots,v_k)=(v_2,\dots,v_k)\qquad\text{and}\qquad \tilde{\gamma}(v)=(v,v_2,\dots,v_k).
\end{gathered}
\end{equation}
Analogously, the Stiefel manifold $W(k,m)$ fibers over $W(k-1,m)$ with fiber $S^{2(m-k)+1}$,
\begin{equation}\label{eq:fib.gen}
\begin{gathered}
\xymatrix{\mathbb{S}^{2(m-k)+1}\cong W(1,m-k+1)\ar@{^{(}->}[r]^(.7){\gamma} & W(k,m)\ar[r]^{p} & W(k-1,m),}\\[5pt]
p(v_1,v_2,\dots,v_k)=(v_2,\dots,v_k)\qquad\text{and}\qquad \gamma(v)=(v,v_2,\dots,v_k).
\end{gathered}
\end{equation}
A complex $k$-frame in $\C^m$ determines a real $2k$-frame in $\R^{2m}$ as follows:
\begin{equation}\label{eq:cpl.fr.re.fr}
(v_1,\dots,v_k)\mapsto(v_1,iv_1,\dots,v_k,iv_k).
\end{equation}
If $(v_1,\dots,v_k)$ is orthonormal, then so is $(v_1,iv_1,\dots,v_k,iv_k)$.
Thus, $\eqref{eq:cpl.fr.re.fr}$ gives a embedding
\begin{equation}\label{eq:comp.real}
 q\colon W(k,m)\to V(2k,2m).
\end{equation}

Let $M$ be a $2(m-k)+1$-dimensional closed oriented manifold. Let
\begin{equation*}
 \phi\colon M\to W(k,m),   
\end{equation*}
be a continuous map.
The map $\phi$ has a well defined degree $\deg(\phi)\in\Z$ given by considering the homomorphism induced in homology
\begin{equation*}
\phi_*\colon H_{2(m-k)+1}(M;\Z)\cong\Z\to H_{2(m-k)+1}\bigl(W\bigl(k,m\bigr);\Z\bigr)\cong\Z,
\end{equation*}
the image of the fundamental class $[M]\in H_{2n-3}(M;\Z)$ of $M$ has the form 
\begin{equation*}
\phi_*([M])=\lambda\cdot \gamma_*([\mathbb{S}^{2(m-k)+1}]),
\end{equation*}
for some integer $\lambda$.
The \textit{degree of $\phi$} is given by 
\begin{equation}\label{eq:deg}
\deg(\phi):=\lambda.
\end{equation}
The following proposition gives a geometric interpretation of the degree of the map $\phi$ (compare with \cite[Lemma~2.5]{Aguilar-Seade-Verjovsky:IVFTIRAS}).

\begin{proposition} \label{proposition:mixedGSVobstruction}
Let $M$ be a $2(m-k)+1$-dimensional, closed oriented manifold. 
Let
\begin{equation*}
 \phi \colon M \to W(k,m),   
\end{equation*}
be a continuous map. 
Let $\gamma \colon \mathbb{S}^{2(m-k)+1} \to  W(k,m)$ be the canonical embedding given in \eqref{eq:fib.gen}. 
Then, there exists a  map
\begin{equation*}
\hat{\phi}\colon M \to \mathbb{S}^{2(m-k)+1},
\end{equation*}
unique up to homotopy, such that $\gamma \circ \hat{\phi}$ is homotopic to $\phi$.
\end{proposition}
\begin{proof} The proof is the same as~\cite[Lemma~2.5]{Aguilar-Seade-Verjovsky:IVFTIRAS}.
\end{proof}	

\begin{remark}\label{rmk:deg.deg}
By Proposition~\ref{proposition:mixedGSVobstruction}, up to homotopy, the map $\phi\colon M\to W\bigl(k,m\bigr)$ can be regarded as a map $\hat{\phi}\colon M\to\mathbb{S}^{2(m-k)+1}$. Its degree $\deg(\hat{\phi})$ is the usual degree of a map between closed oriented manifolds of the same dimension, and, since the homomorphism
$\gamma_*$ induced in homology is an isomorphism we have that $\deg(\phi)=\deg(\hat{\phi})$. Furthermore, let $M_1,\dots,M_r$ be the connected components of $M$ and let $\phi_i=\phi|_{M_i}$ and 
$\hat{\phi}_i=\hat{\phi}|_{M_i}$  for $i=1,\dots,r$. Then $\deg(\phi_i)=\deg(\hat{\phi}_i)$.
\end{remark}

\begin{remark}\label{rmk:deg.v1}
Note that the map $\phi \colon M \to W(k,m)$ has the form $\phi(x)=(v_1(x),\dots,v_k(x))$, where $v_i\colon M\to\mathbb{S}^{2m-1}\subset\C^m$ and $(v_1(x),\dots,v_k(x))$ is a complex orthonormal $k$-frame in $\C^m$.
If we fix $(v_2(x),\dots,v_k(x))$ then $v_1(x)$ is in the unit sphere $\mathbb{S}^{2(m-k)+1}$ in the $m-k+1$-dimensional complex subspace of $\C^m$ orthogonal to the $(k-1)$-dimensional complex subspace of $\C^m$
generated by $v_2(x),\dots,v_k(x)$.  Thus, we can take the map $\hat{\phi}\colon M\to \mathbb{S}^{2(m-k)+1}$ as the map $v_1\colon M\to \mathbb{S}^{2(m-k)+1}$, by \eqref{eq:fib.gen} is clear that $\gamma\circ v_1=\phi$
and by Proposition~\ref{proposition:mixedGSVobstruction} we have $\deg(\phi)=\deg(v_1)$.
\end{remark}

Now let $M$ be a $m-k$-dimensional closed oriented manifold. Let \begin{equation*}
\varphi\colon M\to V(m,k),
\end{equation*}
be a continuous map. We only describe the case that we use, so suppose that $m-k$ is odd.
The map $\varphi$ has a well defined \textit{degree modulo $2$} $\deg_2(\varphi)\in\Z_2$ given by considering the homomorphism induced in homology
\begin{equation*}
\varphi_*\colon H_{m-k}(M;\Z)\cong\Z\to H_{m-k}\bigl(V\bigl(k,m)\bigr);\Z\bigr)\cong\Z_2,
\end{equation*}
and taking
\begin{equation}\label{eq:deg2}
\deg_2(\varphi)=\varphi_*([M])\in\Z_2.
\end{equation}
There is also a result analogous to Proposition~\ref{proposition:mixedGSVobstruction} (see~\cite[Proposition~2.6]{Aguilar-Seade-Verjovsky:IVFTIRAS}). 

\begin{proposition}\label{prop:Vkm.smn}
 Let $M$ be a $m-k$-dimensional, closed oriented manifold. Let
\begin{equation*}
 \varphi \colon M \to V(k,m),   
\end{equation*}
be a continuous map. 
Let $\tilde{\gamma} \colon \mathbb{S}^{m-k} \to  V(k,m)$ be the canonical embedding given in \eqref{eq:fib.gen.real}. 
Then, there exists a map
\begin{equation*}
\hat{\varphi}\colon M \to \mathbb{S}^{m-k},
\end{equation*}
unique up to homotopy, such that $\tilde{\gamma} \circ \hat{\varphi}$ is homotopic to $\varphi$. 
\end{proposition}

\begin{remark}\label{rmk:deg2.deg}
Up to homotopy, the map $\varphi\colon M\to V\bigl(k,m\bigr)$ can be regarded as a map $\hat{\varphi}\colon M\to\mathbb{S}^{m-k}$.
Its degree $\deg(\hat{\varphi})$ is the usual degree of a map, and, since the homomorphism
$\tilde{\gamma}_*$ induced in homology is the reduction modulo $2$, we have that
\begin{equation*}
    \deg_2(\varphi)=\deg(\hat{\varphi}) \mod{2}.
\end{equation*}
Furthermore, let $M_1,\dots,M_r$ be the connected components of $M$ and let $\varphi_i=\varphi|_{M_i}$ and 
$\hat{\varphi}_i=\hat{\varphi}|_{M_i}$  for $i=1,\dots,r$. Then $\deg(\varphi_i)=\deg(\hat{\varphi}_i)\mod 2$.
\end{remark}

\begin{remark}\label{rmk:deg2.v1}
The map $\varphi \colon M \to V(k,m)$ has the form $\phi(x)=(v_1(x),\dots,v_k(x))$, we can take $\hat{\varphi}\colon M\to \mathbb{S}^{m-k}$ as the map 
$v_1\colon M\to \mathbb{S}^{m-k}$ and by Remark~\ref{rmk:deg2.deg} we have $\deg_2(\varphi)=\deg(v_1)\mod 2$.
\end{remark}

\begin{remark}
In \cite[\S2]{Aguilar-Seade-Verjovsky:IVFTIRAS} $\deg_2(\varphi)$ is defined in an equivalent way, by evaluating a \textit{characteristic element} $u\in H^{m-k}\bigl(V\bigl(k,m\bigr);\Z_2\bigr)$
on $\varphi_*([M])$, that is $\deg_2(\varphi)=\langle u,\varphi_*([M])\rangle$.
\end{remark}

\subsection{The GSV-index}\label{ssec:GSV}
Let $f \colon \mathbb{C}^n \to \mathbb{C}$ be a holomorphic function with an isolated critical point at the origin. Let $L_f$ be the link of $f$. We have that $L_f$ is a compact oriented smooth manifold of dimension $2n-3$.
Also consider a continuous vector field $v$ on $V_f$ with an isolated singularity at the origin.
The (conjugate) of the gradient vector field $\overline{\nabla f}$ is normal to $V_f^*$ for the standard Hermitian 
inner product in $\mathbb{C}^n$. Therefore the ordered set of vectors $(v(\mathbf{z}),\overline{\nabla f(\mathbf{z})})$ is a complex $2$-frame in $\C^n$ at each point $\mathbf{z}$ in $V_f^*$, 
and up to homotopy, it can be assumed to be orthonormal. Hence, we get a continuous map from $V_f^*$ to $W(2,n)$ which we can restrict to the link to obtain a map
\begin{equation}\label{eq:GSV-map.complex}
\tilde{\phi}_v:= (v,\overline{\nabla f}) \colon L_f \to W(2,n).
\end{equation}

\begin{definition}[\cite{Gomez-Mont-Seade-Verjovsky:IHFIS}]
The \textit{GSV-index} $\Ind[GSV](v,\mathbf{0})$ of $v$ is given by $\Ind[GSV](v,\mathbf{0}):=\deg(\tilde{\phi}_v)$.
\end{definition}

Sometimes we shall need to consider the $2$-frame $(v,\overline{\nabla f})$ with the reverse order $(\overline{\nabla f},v)$.
The next Lemma tell us that that both frames give us the same index.

\begin{lemma}
\label{lemma:CambioW}
Let $M$ be a $2n-3$-dimensional, closed oriented manifold.
Let
\begin{equation*}
 \phi_1 \colon M \to W(2,n) \quad \text{and} \quad    \phi_2 \colon M \to W(2,n),
\end{equation*}
be continuous maps given by
\begin{equation*}
    \phi_1(m)=(v_1(m), v_2(m)) \quad \text{and} \quad \phi_2(m)=(v_2(m), v_1(m)).
\end{equation*}
Then, $\deg(\phi_1)=\deg(\phi_2).$
\end{lemma}
\begin{proof}
    Consider the following map,
    \begin{align*}
        \lambda \colon W(2,n) &\to W(2,n) \\
        (v_1,v_2) &\mapsto (v_2,v_1)
    \end{align*}
Since $\phi_2=\lambda \circ \phi_1$, it is enough to prove that $\lambda$ is homotopic to the identity. We prove this assertion. Denote by $\mathrm{Id}_{n-2}$ the identity matrix of order $n-2$,
\begin{equation*}
    P=\begin{pmatrix}
0 & 1 \\
1 & 0
\end{pmatrix} \quad \text{and} \quad \Lambda=\begin{pmatrix}
P & 0 \\
0 & \mathrm{Id}_{n-2}
\end{pmatrix}.
\end{equation*}
Note that multiplication on the right by $\Lambda$ gives a map from $\mathrm{U}(n)$ to itself. As usual, denote by $\{e_1,\dots,e_n\}$ the standar complex basis of $\C^n$. The following map
\begin{align*}
    \alpha\colon  \mathrm{U}(n) &\to W(2,n) \\
    U &\mapsto \{Ue_1,Ue_2\}
\end{align*}
induces the following conmutative diagram
\begin{equation}
\label{eq:ChangeCoef.1}
\begin{diagram}[width=7em]
 \mathrm{U}(n) &\rTo^{\alpha}& W(2,n)\\
 \dTo^{\Lambda} &\circlearrowleft&\dTo^{\lambda}\\
\mathrm{U}(n) &\rTo^{\alpha}& W(2,n)
\end{diagram}
\end{equation}
Since $P \in \mathrm{U}(2)$ and $\mathrm{U}(2)$ is path-connected, we can find $r\colon [0,1] \to \mathrm{U}(2)$ a continuous path from the identiy matrix to $P$, i.e., $r(0)=\mathrm{Id}_2$ and $r(1)=P$.
Using the path $r$, we have the following homotopy from the identity to the map given by the multiplication by $\Lambda$
\begin{align*}
    H\colon \mathrm{U}(n) \times [0,1] &\to \mathrm{U}(n) \\
    (U,t) &\mapsto U\begin{pmatrix}
r(t) & 0 \\
0 & \mathrm{Id}_{n-2}
\end{pmatrix}
\end{align*}
Since $\mathrm{U}(n)/\mathrm{U}(n-2) \cong W(2,n)$ and the homotopy $H$ is the idendity in the last $n-2$ columns of any matrix $U \in \mathrm{U}(n)$. Then, $H$ induces a homotopy between
$\lambda$ and the identity map. This proves the lemma.
\end{proof}

\subsection{The real GSV-index}
Let $(V, \mathbf{0})$ be the germ of a complete intersection with an isolated singularity at the origin, defined by a real analytic function
\begin{equation*}
f=(f_1,f_2)\colon U \subset \R^{2n} \to \R^2,
\end{equation*}
where $U$ is an open neighborhood of the origin in $\R^{2n}$. Let $L_f$ be the link of $f$ at $0$. We have that $L_f$ is a compact oriented smooth manifold of dimension $2n-3$. Let $v$ be a continuous vector field on $V_f$ which is singular only at the origin. Let $(\nabla {f_1},\nabla {f_2})$ be the gradient vector fields of $(f_1,f_2)$. Since $V_f$ is an  isolated complete intersection singularity, the vectors $\{\nabla {f_1},\nabla {f_2}\}$
are linearly independent everywhere on $V_f^*$. Hence, up to homotopy, one has a continuous map 
\begin{equation}\label{eq:GSV-map.real}
\tilde{\varphi}_v:= (v,\nabla {f_1},\nabla {f_2}) \colon L_f \to V(3,2n).
\end{equation}

\begin{definition}[\cite{Aguilar-Seade-Verjovsky:IVFTIRAS}]
The \emph{real GSV-index} of $v$ at the origin is given by
\begin{equation*}
\Ind[\R GSV](v,0):=\deg_2(\tilde{\varphi}_v).
\end{equation*}
Let $L_1,\dots,L_r$ be the connected components of the link $L_f$. The \textit{multi-index} $\Ind[\R GSV]^{\mathrm{multi}}(v,\mathbf{0})$ of the vector field $v$ is defined in~\cite{Aguilar-Seade-Verjovsky:IVFTIRAS} by
\begin{equation*}
\Ind[\R GSV]^{\mathrm{multi}}(v,\mathbf{0}):=\left(\deg_2(\tilde{\varphi}_{v_1}),\dots,\deg_2(\tilde{\varphi}_{v_r})\right),
\end{equation*}
where $\tilde{\varphi}_{v_j}$ is the restriction of $\tilde{\varphi}_v$ to the component $L_j$. Notice that one has
\begin{equation*}
\Ind[\R GSV](v,\mathbf{0})=\deg_2(\tilde{\varphi}_{v_1})+\dots+\deg_2(\tilde{\varphi}_{v_r}).
\end{equation*}
\end{definition}
For a complete discussion see \cite{Aguilar-Seade-Verjovsky:IVFTIRAS} or \cite[Chapter~4]{Brasselet-etal:VFSV}.

\section{Mixed functions}
\label{sec:mixed}
Let $\mu=(\mu_1,\dots,\mu_n)$ and $\nu=(\nu_1,\dots,\nu_n)$ with $\mu_j, \nu_j \in \N\cup\{0\}$. Set $\mathbf{z}^\mu=z_1^{\mu_1}\dots z_n^{\mu_n}$ and 
$\bar{\mathbf{z}}^\nu=\bar{z}_1^{\nu_1}\dots\bar{z}_n^{\nu_n}$. 
Consider a complex valued function $f\colon\mathbb{C}^n\to\mathbb{C}$ expanded in a convergent power series of variables $\mathbf{z}$ and $\bar{\mathbf{z}}$,
\begin{equation*}
f(\mathbf{z})=\sum_{\mu,\nu}c_{\mu,\nu}\mathbf{z}^{\mu}\bar{\mathbf{z}}^{\nu}.
\end{equation*}
We call $f$ a \emph{mixed analytic function} (or a \emph{mixed polynomial}, if $f$ is a polynomial). 

We consider $f$ as a function $f\colon\R^{2n}\to\R^2$ in the $2n$ real variables $(\mathbf{x},\mathbf{y})$ writing 
\begin{equation}\label{eq:f.re.im}
f(\mathbf{z}_\R)=f(\mathbf{x},\mathbf{y})= g(\mathbf{x},\mathbf{y}) + i h(\mathbf{x},\mathbf{y}),
\end{equation}
where $g,h\colon \C^n \cong \mathbb{R}^{2n} \to \mathbb{R}$ are real analytic functions. 
Let $\bar{f}$ be the conjugate of $f$, i.e.,
\begin{equation*}
\bar{f}(\mathbf{z})=\sum_{\mu,\nu}\bar{c}_{\mu,\nu}\bar{\mathbf{z}}^{\mu}\mathbf{z}^{\nu}.
\end{equation*}
Hence we can get $g$ and $h$ by
\begin{equation}\label{eq:gyh}
g=\frac{1}{2}\left(f+\bar{f}\right),\qquad\text{and}\qquad h=\frac{1}{2i}\left(f-\bar{f}\right).
\end{equation}
As usual, we define the real gradients of $g$ and $h$ by 
\begin{equation}\label{eq:nablas}
\nabla g = \left (\frac{\partial g}{\partial x_1},\frac{\partial g}{\partial y_1}, \dots, \frac{\partial g}{\partial x_n},\frac{\partial g}{\partial y_n}\right),\qquad 
\nabla h = \left (\frac{\partial h}{\partial x_1},\frac{\partial h}{\partial y_1}, \dots, \frac{\partial h}{\partial x_n},\frac{\partial h}{\partial y_n}\right).
\end{equation}
Hence, the differential of $f$ at a point $(\mathbf{x},\mathbf{y})$ is given by the real $2\times 2n$ matrix
\begin{equation}\label{eq:differential}
Df_{(\mathbf{x},\mathbf{y})}=\begin{pmatrix}
                                            \nabla g(\mathbf{x},\mathbf{y})\\
                                            \nabla h(\mathbf{x},\mathbf{y})
                                           \end{pmatrix}=\begin{pmatrix}
                              \frac{\partial g}{\partial x_1}(\mathbf{x},\mathbf{y}) & \frac{\partial g}{\partial y_1}(\mathbf{x},\mathbf{y}) & \dots & \frac{\partial g}{\partial x_n}(\mathbf{x},\mathbf{y}) & \frac{\partial g}{\partial y_n}(\mathbf{x},\mathbf{y})\\[5pt]
                              \frac{\partial h}{\partial x_1}(\mathbf{x},\mathbf{y}) & \frac{\partial h}{\partial y_1}(\mathbf{x},\mathbf{y}) & \dots & \frac{\partial h}{\partial x_n}(\mathbf{x},\mathbf{y}) & \frac{\partial h}{\partial y_n}(\mathbf{x},\mathbf{y})
                             \end{pmatrix}.
\end{equation}
Recall that for any real analytic function $k\colon\R^{2n}\to\R$ we have
\begin{equation}\label{eq:wirtinger}
\frac{\partial k}{\partial z_j}=\frac{1}{2} \left (\frac{\partial k}{\partial x_j} - i\frac{\partial k}{\partial y_j} \right ), \qquad
\frac{\partial k}{\partial \bar{z}_j}=\frac{1}{2}\left(\frac{\partial k}{\partial x_j} + i\frac{\partial k}{\partial y_j} \right ).
\end{equation}
We also have
\begin{equation}\label{eq:partial.f.z}
\frac{\partial f}{\partial z_j}=\frac{\partial g}{\partial z_j} + i\frac{\partial h}{\partial z_j}\;, \qquad 
\frac{\partial f}{\partial \bar{z}_j}=\frac{\partial g}{\partial \bar{z}_j} + i\frac{\partial h}{\partial \bar{z}_j}.
\end{equation}
Using \eqref{eq:partial.f.z} and \eqref{eq:wirtinger} we get
\begin{equation}\label{eq:df.dg.dh}
\frac{\partial f}{\partial z_j}=\frac{1}{2}\biggl(\frac{\partial g}{\partial x_j}+\frac{\partial h}{\partial y_j} + i \Bigl(\frac{\partial h}{\partial x_j}-\frac{\partial g}{\partial y_j}\Bigr)\biggr),\qquad
\frac{\partial f}{\partial \bar{z}_j}=\frac{1}{2}\biggl(\frac{\partial g}{\partial x_j}-\frac{\partial h}{\partial y_j} + i \Bigl(\frac{\partial h}{\partial x_j}+\frac{\partial g}{\partial y_j}\Bigr)\biggr).
\end{equation}
From the definition it is easy to see that
\begin{equation}\label{eq:partials}
 \frac{\partial \bar{f}}{\partial z_i}=\overline{\frac{\partial f}{\partial \bar{z}_i}},\qquad\text{and}\qquad\frac{\partial \bar{f}}{\partial \bar{z}_i}=\overline{\frac{\partial f}{\partial z_i}}.
\end{equation}
Adding and substracting equations \eqref{eq:wirtinger} we get
\begin{equation}\label{eq:wirt.inv}
\frac{\partial k}{\partial x_j}=\frac{\partial k}{\partial z_j}+\frac{\partial k}{\partial \bar{z}_j},\qquad 
\frac{\partial k}{\partial y_j}=i\biggl(\frac{\partial k}{\partial z_j}-\frac{\partial k}{\partial \bar{z}_j}\biggr).
\end{equation}
Using \eqref{eq:gyh}, \eqref{eq:partials} and \eqref{eq:wirt.inv} for $j=1,\dots,n$ we obtain
\begin{align*}
\frac{\partial g}{\partial x_j}&=\re\biggl(\overline{\frac{\partial f}{\partial z_j}}+\frac{\partial f}{\partial \bar{z}_j}\biggr), &
\frac{\partial g}{\partial y_j}&=\im\biggl(\overline{\frac{\partial f}{\partial z_j}}+\frac{\partial f}{\partial \bar{z}_j}\biggr), \\
\frac{\partial h}{\partial x_j}&=\re\Biggl(i\biggl(\overline{\frac{\partial f}{\partial z_j}}-\frac{\partial f}{\partial \bar{z}_j}\biggr)\Biggr), &
\frac{\partial h}{\partial y_j}&=\im\Biggl(i\biggl(\overline{\frac{\partial f}{\partial z_j}}-\frac{\partial f}{\partial \bar{z}_j}\biggr)\Biggr). 
\end{align*}
Hence, to the vectors $\nabla g(\mathbf{x},\mathbf{y})$ and $\nabla h(\mathbf{x},\mathbf{y})$ in $\R^{2n}$ given in \eqref{eq:nablas}, correspond the vectors in $\C^n$
\begin{align}
\nabla g(\mathbf{z})&=\nabla g(\mathbf{x},\mathbf{y})_\C=\Biggl(\overline{\frac{\partial f}{\partial z_1}}+\frac{\partial f}{\partial \bar{z}_1},\dots,\overline{\frac{\partial f}{\partial z_n}}+\frac{\partial f}{\partial \bar{z}_n}\Biggr),\label{eq:nablag.z}\\
\nabla h(\mathbf{z})&=\nabla h(\mathbf{x},\mathbf{y})_\C=\Biggl(i\biggl(\overline{\frac{\partial f}{\partial z_1}}-\frac{\partial f}{\partial \bar{z}_1}\biggr),\dots,i\biggl(\overline{\frac{\partial f}{\partial z_n}}-\frac{\partial f}{\partial \bar{z}_n}\biggr)\Biggr).\label{eq:nablah.z}
\end{align}
Following Oka \cite{Oka:TPWHH} set
\begin{equation*}
\mathrm{d} f(\mathbf{z})= \left ( \frac{\partial f}{\partial z_1}(\mathbf{z}), \dots , \frac{\partial f}{\partial z_n}(\mathbf{z}) \right )\;, \quad \quad \bar{\mathrm{d}}f(\mathbf{z})= \left ( 
\frac{\partial 
f}{\partial \bar{z}_1}(\mathbf{z}), \dots , \frac{\partial f}{\partial \bar{z}_n}(\mathbf{z}) \right )\;.
\end{equation*}
From \eqref{eq:nablag.z} and \eqref{eq:nablah.z} it is easy to see that
\begin{equation}\label{eq:nabla.d}
\nabla g(\mathbf{z})=\overline{\mathrm{d}f(\mathbf{z})}+\bar{\mathrm{d}}f(\mathbf{z}),\qquad\nabla h(\mathbf{z})=i\overline{\mathrm{d}f(\mathbf{z})}-i\bar{\mathrm{d}}f(\mathbf{z}).
\end{equation}
From \eqref{eq:nabla.d} we get
\begin{equation}
\label{eq:df.usando.g.h}
\overline{\mathrm{d}f(\mathbf{z})}=\frac{1}{2}\bigl(\nabla g(\mathbf{z})-i\nabla h(\mathbf{z})\bigr),\qquad \bar{\mathrm{d}}f(\mathbf{z})=\frac{1}{2}\bigl(\nabla g(\mathbf{z})+i\nabla h(\mathbf{z})\bigr).
\end{equation}
From \eqref{eq:nabla.d} the following proposition is immediate.

\begin{proposition}\label{prop:cplx.partials}
Let $f=g+ih\colon\C^n\to\C$ be a mixed function. Then
\begin{enumerate}[1.]
 \item The vector $\mathrm{d}f(\mathbf{z})$ has all components equal to zero if and only if $\nabla h(\mathbf{z})=-i\bar{\mathrm{d}}f(\mathbf{z})=-i\nabla g(\mathbf{z})$.\label{it:df=0}
 \item The vector $\bar{\mathrm{d}}f(\mathbf{z})$ has all components equal to zero if and only if $\nabla h(\mathbf{z})=i\overline{\mathrm{d}f(\mathbf{z})}=i\nabla g(\mathbf{z})$.\label{it:bdf=0}
\end{enumerate}
\end{proposition}

The following proposition is a generalization of \cite[Proposition~1]{Oka:TPWHH}.

\begin{proposition}\label{prop:df.bdf.ng.nh}
Let $f=g+ih\colon\C^n\to\C$ be a mixed function. The following conditions are equivalent:
\begin{enumerate}
 \item The vectors $\overline{\mathrm{d}f(\mathbf{z})}$ and $\bar{\mathrm{d}}f(\mathbf{z})$ are linearly dependent over the complex numbers.\label{it:df.bdf.ld}
 \item The vectors $\nabla g(\mathbf{z})$, $\nabla h(\mathbf{z})$ and $i\nabla h(\mathbf{z})$ are linearly dependent over the real numbers.\label{it:ng.nh.inh}
 \item The vectors $\nabla g(\mathbf{z})$ and $\nabla h(\mathbf{z})$ are linearly dependent over the complex numbers.\label{it:ng.nh.cld}
 \item The vectors $\nabla h(\mathbf{z})$, $\nabla g(\mathbf{z})$ and $i\nabla g(\mathbf{z})$ are linearly dependent over the real numbers.\label{it:nh.ng.ing}
\end{enumerate}
\end{proposition}

\begin{proof}
\paragraph{$\eqref{it:df.bdf.ld}\Rightarrow\eqref{it:ng.nh.inh}$}
If $\mathrm{d}f(\mathbf{z})=0$ or $\bar{\mathrm{d}}f(\mathbf{z}) = 0$, then by \eqref{eq:df.usando.g.h} we are done.
Thus, we may assume  $\overline{\mathrm{d}f(\mathbf{z})}$ and $\bar{\mathrm{d}}f(\mathbf{z})$ are different from zero. By hypothesis, the vectors are linearly dependent over the complex numbers, hence there exists a complex number $\alpha =a+ib$ such that
\begin{equation*}
 \overline{\mathrm{d}f(\mathbf{z})}=\alpha\bar{\mathrm{d}}f(\mathbf{z}).
\end{equation*}
Using \eqref{eq:df.dg.dh} in this equality we get for each $j=1,\dots,n$
the following pair of equations obtained by Oka in the proof of \cite[Proposition~1]{Oka:TPWHH}:
\begin{align}
(1-a)\frac{\partial g}{\partial x_j} + b \frac{\partial g}{\partial y_j} &= -b \frac{\partial h}{\partial x_j} - (1+a) \frac{\partial h}{\partial y_j},\label{eq:eq1}\\
-b \frac{\partial g}{\partial x_j} + (1-a)\frac{\partial g}{\partial y_j}&= (a+1)\frac{\partial h}{\partial x_j} - b \frac{\partial h}{\partial y_j}.\label{eq:eq2}
\end{align}
Solving \eqref{eq:eq1} and \eqref{eq:eq2} assuming that $\alpha\neq1$ we get
\begin{equation}\label{eq:ng.nh.inh}
\nabla g(\mathbf{z})=\frac{-2b}{(1-a)^2+b^2}\nabla h(\mathbf{z})+\frac{(1-\abs{\alpha}^2)}{(1-a)^2+b^2}i\nabla h(\mathbf{z}).
\end{equation}
If $\alpha=1$, that is, $a=1$ and $b=0$, equations \eqref{eq:eq1} and \eqref{eq:eq2} imply that $\nabla h=0$ and the linear dependence is obvious.

\paragraph{$\eqref{it:ng.nh.inh}\Rightarrow\eqref{it:ng.nh.cld}$}
By hypothesis, there exist real numbers $t_1,t_2$ and $t_3$, not all zero, such that
\begin{equation*}
t_1\nabla g(\mathbf{z})+t_2\nabla h(\mathbf{z})+t_3i\nabla h(\mathbf{z})=0.
\end{equation*}
Let $\alpha=t_2+it_3$. Under this identification, the previous equality is
\begin{equation*}
t_1\nabla g(\mathbf{z})+\alpha\nabla h(\mathbf{z})=0.
\end{equation*}
Therefore, the vectors $\nabla g(\mathbf{z})$ and $\nabla h(\mathbf{z})$ are linearly dependent over the complex numbers.

\paragraph{$\eqref{it:ng.nh.cld}\Rightarrow\eqref{it:nh.ng.ing}$}
If $\nabla h(\mathbf{z})=0$ or $\nabla g(\mathbf{z})=0$, then the statement is true. Thus, we may assume that both vectors are different from zero. Suppose $\nabla h(\mathbf{z})=\beta \nabla g(\mathbf{z})$ for some complex number $\beta=c+id$, then we have 
\begin{equation}\label{eq:nh.ng.ing}
\nabla h(\mathbf{z})=c \nabla g(\mathbf{z})+di\nabla g(\mathbf{z}).
\end{equation}
Therefore the vectors $\nabla h(\mathbf{z})$, $\nabla g(\mathbf{z})$ and $i\nabla g(\mathbf{z})$ are linearly dependent over the real numbers.

\paragraph{$\eqref{it:nh.ng.ing}\Rightarrow\eqref{it:df.bdf.ld}$}
Suppose that there exist real numbers $t_1,t_2$ and $t_3$, not all zero, such that
\begin{equation}\label{eq:nh.ng.ing.1}
t_1\nabla h(\mathbf{z})+t_2 \nabla g(\mathbf{z})+t_3i\nabla g(\mathbf{z})=0.
\end{equation}
Suppose $t_1=0$. By equation~\eqref{eq:nh.ng.ing.1} we get
\begin{equation*}
\nabla g(\mathbf{z})(t_2+it_3)=0.
\end{equation*}
Since $t_2+it_3$ is not zero, then $\nabla g(\mathbf{z})=0$. Therefore by equation \eqref{eq:nabla.d} the vectors $\overline{\mathrm{d}f(\mathbf{z})}$ and $\bar{\mathrm{d}}f(\mathbf{z})$ are linearly dependent over the complex numbers.

Now suppose $t_1\neq 0$. By equation \eqref{eq:nh.ng.ing.1} and changing the variables we get
\begin{equation}\label{eq:nh.ng.ing.2}
\nabla h(\mathbf{z})=c \nabla g(\mathbf{z})+di\nabla g(\mathbf{z}).
\end{equation}
Using \eqref{eq:nh.ng.ing.2} in \eqref{eq:df.usando.g.h} we get
\begin{equation}
\label{eq:df.bardf.1}
2\overline{\mathrm{d}f(\mathbf{z})}=\bigl((1+d)-ci\bigr) \nabla g(\mathbf{z}),\qquad
2\bar{\mathrm{d}}f(\mathbf{z})=\bigl((1-d)+ci\bigr)\nabla g(\mathbf{z}).
\end{equation}
If $(1-d)+ci\neq0$, then by equation~\eqref{eq:df.bardf.1} we get \begin{equation}\label{eq:df.bdf.li}
\overline{\mathrm{d}f(\mathbf{z})}=\frac{(1+d)-ci}{(1-d)+ci}\bar{\mathrm{d}}f(\mathbf{z}).
\end{equation}
Hence $\overline{\mathrm{d}f(\mathbf{z})}$ and $\bar{\mathrm{d}}f(\mathbf{z})$ are linearly dependent over the complex numbers.
If $(1-d)+ci=0$, then $c=0$ and $d=1$. By~\eqref{eq:nh.ng.ing.1} we get $\nabla h(\mathbf{z})=i\nabla g(\mathbf{z})$. Thus, by Proposition~\ref{prop:cplx.partials}-\ref{it:bdf=0} this is
equivalent to have $\bar{\mathrm{d}}f(\mathbf{z})=0$. Now the linear dependence of $\overline{\mathrm{d}f(\mathbf{z})}$ and $\bar{\mathrm{d}}f(\mathbf{z})$ is obvious.
\end{proof}

By Proposition~\ref{prop:df.bdf.ng.nh}, we get the following useful criterium due to Oka.
\begin{proposition}[Oka's Criterium {\cite[Proposition~1]{Oka:TPWHH}}]\label{lema oka}
Let $f=g+ih\colon\C^n\to\C$ be a mixed function. Let $\mathbf{z} \in \C^n$. The following two conditions are equivalent,  
\begin{enumerate}
\item The vectors $\nabla g(\mathbf{z})$ and $\nabla h(\mathbf{z})$ are linearly dependent over $\R$.\label{it:oka.crit1}
\item There exists a complex number $\alpha \in \mathbb{S}^1$ such that $\overline{\mathrm{d} f(\mathbf{z})}=\alpha \bar{\mathrm{d}}f(\mathbf{z})$.\label{it:oka.crit2}
\end{enumerate}
\end{proposition}

\subsection{Polar weighted homogeneous polynomials}

An important class of mixed functions, that we shall use in Examples~\ref{exa:strong.polar} and \ref{example:fallaLI},
is given by polar weighted homogeneous polynomials, which are real analytic maps that generalize complex weighted homogeneous polynomials.
They are mixed polynomials which are weighted homogeneous with respect to an $\mathbb{R}^+$-action and also with respect to a $\mathbb{S}^1$-action.
Let $p_1,\dots,p_n$ and $q_1,\dots,q_n$ be non-zero integers such that
\begin{equation*}
 \gcd (p_1, \dots , p_n) = 1\quad \text{and} \quad \gcd(q_1, \dots, q_n) = 1.
\end{equation*}
Let $w \in \C^*$ written in its polar form
$w=t\tau$, with $t \in \mathbb{R}^+$ and $\tau \in \mathbb{S}^1$. A \emph{polar $\C^*$-action} on $\C^n$ with \emph{radial weights} $(p_1, \dots , p_n)$ and \emph{angular weights} $(q_1,\dots, q_n)$
is given by:
\begin{equation}\label{eq:polar-action}
 t\tau \bullet \mathbf{z} = (t^{p_1}\tau^{q_1}z_1, \dots , t^{p_n}\tau^{q_n}z_n)\;.
\end{equation}

\begin{definition}
A mixed function $f\colon \C^n \to \C$ is \emph{polar weighted homogeneous\/} if there
exists $p_1, \dots , p_n$ positive integers, $q_1,\dots, q_n$ non-zero integers, $a, c$ positive integers, and a polar $\C^*$-action given by \eqref{eq:polar-action} such that $f$ satisfies the
following functional equation
\begin{equation}\label{eq:func.eq}
f(t\tau \bullet \mathbf{z}) = t^a\tau^c f(\mathbf{z})\;.
\end{equation}
We say that the polar weighted homogeneous function $f$ has \emph{radial weight type} $(p_1,\dots,p_n;a)$ and \emph{angular weight type} $(q_1,\dots,q_n;c)$.
Following Oka in \cite{Oka:OMPC} we say that $f$ is \textit{strongly polar weighted homogeneous} if $p_j=q_j$ for all $j=1,\dots, n$.
\end{definition}

Polar weighted homogeneous polynomials were defined in \cite{Cisneros-Molina:Join} inspired by the definition of the so-called twisted Brieskorn-Pham polynomials defined in \cite{Ruas-Seade-Verjovsky:RSMF,Seade:OBDAHVF}.
They have been intensively studied by Oka \cite{Oka:TPWHH,Oka:NDMF,Oka:MPCPD1,Oka:MBV,Oka:OMPC,Oka:MFSPWHFT}
and other authors \cite{Blanloeil-Oka:TSPWHL,Cisneros-Romano:CIPWHS,Cisneros-Romano:REMSMF,HenandezDLC-LopezDM:SFIS,Inaba:OEMNCMP,Inaba-etal:TMHCT}.

\subsection{Vector fields}
In this subsection we define the notion of vector field in our context. Let $f\colon \C^n \to \C$ be a mixed function with an isolated singularity at the origin.

\begin{definition}
By a (continuous, smooth or holomorphic) \textit{vector field} on $V_f$ we understand the restriction $v$ to $V_f$ of a (continuous, smooth or holomorphic) vector field on a neighborhood of $V_f$ in $\C^n$, 
for which $V_f$ is an invariant set. In other words, if an integral curve of the vector field intersects $V_f$, then it is contained in $V_f$.
\end{definition}

 \begin{definition}
	Let $u$ and $v$ be vector fields on $V_f$ that vanish only at the origin. We say that $u$ and $v$ are \emph{homotopic} if there exists a continuous $1$-parameter 
	family $w_t$ of vector fields on $V_f$ and $t \in [0,1]$, such that $w_0=u$, $w_1=v$ and, for each $t$, the vector field $w_t$ vanish only at the origin. 
	We denote by $\Theta(V_f,\mathbf{0})$ the set of homotopy classes of such vector fields.
\end{definition}

\begin{definition}
We say that a vector field $v$ on $V_f$ is a \textit{complex vector field} if $iv$ is also a vector field on $V_f$.
\end{definition}

\begin{proposition}\label{prop:mixed.vf}
Let $f\colon \C^n \to \C$ be a mixed function with an isolated critical point at the origin and  $\mathbf{z}\in V_f^*$. Let $v=(v_1,\dots,v_n)$ be a vector field on $\C^n$. Then, the following are equivalent:
\begin{enumerate}[(i)]
\item The vector $v(\mathbf{z})$ is tangent to $V_f^*$ at $\mathbf{z}$.\label{it:mvf}
\item The following equalities are satisfied:\label{it:eqs.mixed.vf}
\begin{align}
\inpr[\R^{2n}]{\nabla g(\mathbf{z})}{v(\mathbf{z})}&=\re\Biggl(\sum_{j=1}^n\biggl(\overline{\frac{\partial f}{\partial z_j}}+\frac{\partial f}{\partial \bar{z}_j}\biggr)\bar{v}_j\Biggr)=0,\label{eq:eq.vf.g}\\
\inpr[\R^{2n}]{\nabla h(\mathbf{z})}{v(\mathbf{z})}&=\re\Biggl(\sum_{j=1}^ni\biggl(\overline{\frac{\partial f}{\partial z_j}}-\frac{\partial f}{\partial \bar{z}_j}\biggr)\bar{v}_j\Biggr)=0.\label{eq:eq.vf.h}
\end{align}
\item The following equalities are satisfied:\label{it:eqs.df.dbf}
\begin{align}\label{eq:v.df.dbf}
\inpr[\C^n]{\overline{\mathrm{d}f(\mathbf{z})}}{v(\mathbf{z})}=\sum_{j=1}^n\overline{\frac{\partial f}{\partial z_j}}\bar{v}_j&=c+di, &   \inpr[\C^n]{\bar{\mathrm{d}}f(\mathbf{z})}{v(\mathbf{z})}=\sum_{j=1}^n\frac{\partial f}{\partial \bar{z}_j}\bar{v}_j&=-c+di,
\end{align}
for some real numbers $c$ and $d$.
\end{enumerate}
\end{proposition}

\begin{proof}
\paragraph{$\eqref{it:mvf}\Leftrightarrow\eqref{it:eqs.mixed.vf}$}
The vector $v(\mathbf{z})$ is tangent to $V_f^*$ if and only if
\begin{equation}
Df_\mathbf{z}(v(\mathbf{z}))=\begin{pmatrix}
                                            \nabla g(\mathbf{z})\\
                                            \nabla h(\mathbf{z})
                                           \end{pmatrix}\cdot v(\mathbf{z})^t=\begin{pmatrix}
                             0\\0
                            \end{pmatrix},
\end{equation}
but this is equivalent to equations \eqref{eq:eq.vf.g} and \eqref{eq:eq.vf.h}.
\paragraph{$\eqref{it:eqs.mixed.vf}\Rightarrow\eqref{it:eqs.df.dbf}$}
By equation~\eqref{eq:eq.vf.g} we have
\begin{equation*}
\re\Biggl(\sum_{j=1}^n\frac{\partial f}{\partial \bar{z}_j}\bar{v}_j\Biggr)=-\re\Biggl(\sum_{j=1}^n\overline{\frac{\partial f}{\partial z_j}}\bar{v}_j\Biggr)=-c.
\end{equation*}
By equation~\eqref{eq:eq.vf.h} we have
\begin{equation*}
\im\Biggl(\sum_{j=1}^n\frac{\partial f}{\partial \bar{z}_j}\bar{v}_j\Biggr)=-\re\Biggl(\sum_{j=1}^ni\frac{\partial f}{\partial \bar{z}_j}\bar{v}_j\Biggr)
=-\re\Biggl(\sum_{j=1}^ni\overline{\frac{\partial f}{\partial z_j}}\bar{v}_j\Biggr)=\im\Biggl(\sum_{j=1}^n\overline{\frac{\partial f}{\partial z_j}}\bar{v}_j\Biggr)=d.
\end{equation*}
\paragraph{$\eqref{it:eqs.df.dbf}\Rightarrow\eqref{it:eqs.mixed.vf}$}
It is straightforward.
\end{proof}

As a corollary we get.
\begin{corollary}
\label{cor:complex.vf}
Let $f\colon \C^n \to \C$ be a mixed function with an isolated critical point at the origin and  $\mathbf{z}\in V_f^*$. Let $v=(v_1,\dots,v_n)$ be a vector field on $\C^n$. 
Then, the following are equivalent:
\begin{enumerate}[(i)]
\item The vectors $v(\mathbf{z})$ and $iv(\mathbf{z})$ are tangent to $V_f^*$ at $\mathbf{z}$.\label{it:cvf}
\item The following equalities are satisfied:\label{it:eqs.mixed}
\begin{align}\label{eq:mixed}
 \sum_{j=1}^n\biggl(\overline{\frac{\partial f}{\partial z_j}}+\frac{\partial f}{\partial \bar{z}_j}\biggr)\bar{v}_j&=0, & 
\sum_{j=1}^n\biggl(\overline{\frac{\partial f}{\partial z_j}}-\frac{\partial f}{\partial \bar{z}_j}\biggr)\bar{v}_j&=0.
\end{align}
\item The following equalities are satisfied:\label{it:eqs.holo}
\begin{align}\label{eq:holo}
\sum_{j=1}^n\overline{\frac{\partial f}{\partial z_j}}\bar{v}_j&=0, &   \sum_{j=1}^n\frac{\partial f}{\partial \bar{z}_j}\bar{v}_j&=0.
\end{align}
\end{enumerate}
\end{corollary}

\begin{proof}
The proof follows straightforward by Proposition~\ref{prop:mixed.vf}.
\end{proof}

\begin{proposition}\label{prop:mult.orth}
Let $f\colon \C^n \to \C$ be a mixed function with an isolated critical point at the origin and $\mathbf{z}\in V^*$. Let $v=(v_1,\dots,v_n)$ be a continuous vector field on $V$. If the vectors $\overline{\mathrm{d}f(\mathbf{z})}$ and $\bar{\mathrm{d}}f(\mathbf{z})$ are linearly dependent over $\C$,
then
\begin{equation}\label{eq:df.bdf.orto.v}
\inpr[\C^n]{\overline{\mathrm{d}f(\mathbf{z})}}{v(\mathbf{z})}=0,\qquad\text{and}\qquad\inpr[\C]{\bar{\mathrm{d}}f(\mathbf{z})}{v(\mathbf{z})}=0.
\end{equation}
\end{proposition}

\begin{proof}
If $v(\mathbf{z})=0$ it is clear that equations \eqref{eq:df.bdf.orto.v} are satisfied. So, suppose $v(\mathbf{z})\neq0$. Since $\mathbf{z}\in V^*$ the vectors
$\mathrm{d}f(\mathbf{z})$ and $\bar{\mathrm{d}}f(\mathbf{z})$ cannot be both zero. 
By Proposition~\ref{prop:mixed.vf} there exist real numbers $c$ and $d$ such that
\begin{align*}
\inpr[\C^n]{\overline{\mathrm{d}f(\mathbf{z})}}{v(\mathbf{z})}&=\sum_{j=1}^n\overline{\frac{\partial f}{\partial z_j}}(\mathbf{z})\bar{v}_j(\mathbf{z})=c+id,\\
\inpr[\C^n]{\bar{\mathrm{d}}f(\mathbf{z})}{v(\mathbf{z})}&=\sum_{j=1}^n\frac{\partial f}{\partial \bar{z}_j}(\mathbf{z})\bar{v}_j(\mathbf{z})=-c+id.
\end{align*}
By hypothesis, there exists a complex number $\alpha$,  such that $\overline{\mathrm{d}f(\mathbf{z})}=\alpha\bar{\mathrm{d}}f(\mathbf{z})$. Since $\mathbf{z}\in V^*$ by Proposition~\ref{lema oka}
we have that $|\alpha|\neq1$.
Then we have that
\begin{equation}
\label{eq:Prop.c.id}
c+id=\inpr[\C^n]{\overline{\mathrm{d}f(\mathbf{z})}}{v(\mathbf{z})}=\inpr[\C^n]{\alpha\bar{\mathrm{d}}f(\mathbf{z})}{v(\mathbf{z})}=\alpha(-c+id).
\end{equation}
Computing the norm in both sides of the equality~\eqref{eq:Prop.c.id} we get
\begin{equation*}
|c+di| = |\alpha| |-c+di|.    
\end{equation*}
Since $|\alpha|\neq1$, we get $c=d=0$.
\end{proof}

The following example shows the existence of complex vector fields on real analytic varieties given by strongly polar weighted homogeneous polynomials. See~\cite{Cisneros-Molina:Join, Cisneros-Romano:CIPWHS, Oka:TPWHH} for more details about weighted homogeneous polynomials.

\begin{example}\label{exa:strong.polar}
Let $f\colon \C^n \to \C$ be a strongly polar weighted homogeneous polynomial. Let $v_{\R^+}$ and $v_{\mathbb{S}^1}$ be the radial and the angular vector fields, respectively. 
Since $f$ is a strongly polar weighted homogeneous polynomial we have that $iv_{\R^+}=v_{\mathbb{S}^1}$. If $\lambda\colon\C^n\to\C$ is a
(continuous, smooth or holomorphic) function, then 
\begin{equation*}
\lambda(\mathbf{z)}v_{\R^+}(\mathbf{z)},\qquad\text{for $\mathbf{z}\in V$,}
\end{equation*}
is a complex vector field on $V$. See \cite{Cisneros-Romano:CIPWHS} and \cite{Oka:SMPCC} for specific examples of strongly polar weighted homogeneous polynomial.
\end{example}

\section{The mixed GSV-index}
\label{sec:TheMGSV}
Let $f\colon \C^n \to \C$ be a mixed function. Suppose that  $\mathbf{0}$ is an isolated singular point of $V_f$. Let $v$ be a continuous vector field on $V_f$ which is non-zero on the set of regular points $V_f^*$ of $V_f$. In this section we generalize the GSV-index for the vector field $v$. The idea is to assign to each point $\mathbf{z}$ in $V_f^*$ a complex $2$-frame in $\C^n$.
Let $\mathbf{z}\in V_f^*$ and consider the vector $v(\mathbf{z})$, the following lemma tells us when we can add a vector $u\in\C^n$ to get a complex $2$-frame $(u,v(\mathbf{z}))_\C$ in $\C^n$.

\begin{lemma}\label{lem:NF.li}
Let $R$ be the complex line generated by $v(\mathbf{z})$. Decompose $\C^n$ as the direct sum 
\begin{equation*}
 \C^n=R\oplus R^\perp,   
\end{equation*}
where $R^\perp$ is the orthogonal complement of $R$
with respect to the Hermitian product $\inpr[\C^n]{\ }{\ }$ on $\C^n$. Let $P_{R}\colon \C^n\to R$ and $P_{R^\perp}\colon \C^n\to R^\perp$ be the corresponding orthogonal projections.
Let $u\in\C^n$ be different from the zero vector. Then, the following are equivalent:
\begin{enumerate}
 \item The vector $u$ is not a multiple of $v(\mathbf{z})$.\label{it:no.mult}
 \item The projection $P_{R^\perp}(u)\neq0$.\label{it:no.zero}
\end{enumerate}
\end{lemma}

\begin{proof}
  \paragraph{$\eqref{it:no.mult}\Leftrightarrow\eqref{it:no.zero}$} We have that $u$ decomposes as $u=P_R(u)+P_{R^\perp}(u)$. Hence $u$ is not a multiple of $v(\mathbf{z})$ if and only if $P_{R^\perp}(u)\neq0$.
\end{proof}

Now, we want to see when $(\nabla g(\mathbf{z}),v(\mathbf{z}))_\C$ is a complex $2$-frame in $\C^n$. 
We only need to check if $\nabla g(\mathbf{z})$ is a complex multiple of $v(\mathbf{z})$ or not using Lemma~\ref{lem:NF.li}.
The following results decompose $\nabla g(\mathbf{z})$ with respect to $R$ and $R^\perp$.

\begin{lemma}\label{lem:u}
Let $R$, $R^\perp$, $P_{R}$ and $P_{R^\perp}$ be as in Lemma~\ref{lem:NF.li}. 
\begin{enumerate}
 \item The projection $P_{R}\bigl(\nabla g(\mathbf{z})\bigr)$ is the vector\label{it:u}
\begin{equation}\label{eq:u=aev}
u(\mathbf{z})=\Biggl(\frac{1}{\norm{v(\mathbf{z})}^2}\sum_{i=1}^n\biggl(\overline{\frac{\partial f}{\partial z_i}}(\mathbf{z})+\frac{\partial f}{\partial \bar{z}_i}(\mathbf{z})\biggr)\bar{v}_i(\mathbf{z})\Biggr)v(\mathbf{z}).
\end{equation}

\item The projection $P_{R^\perp}\bigl(\nabla g(\mathbf{z})\bigr)$ is the vector\label{it:w}
\begin{equation*}
w(\mathbf{z})=\nabla g(\mathbf{z})-u(\mathbf{z}).
\end{equation*}
\end{enumerate}
\end{lemma}

\begin{proof}
Firstly notice that $\inpr[\C^n]{v(\mathbf{z})}{v(\mathbf{z})}=\norm{v(\mathbf{z})}^2$. The component of $\nabla g(\mathbf{z})$ in the direction of $v(\mathbf{z})$
is given by
\begin{equation}\label{eq:proj.on.ev}
\begin{split}
\frac{\inpr[\C^n]{\nabla g(\mathbf{z})}{v(\mathbf{z})}}{\inpr[\C^n]{v(\mathbf{z})}{v(\mathbf{z})}}
&=\frac{1}{\norm{v(\mathbf{z})}^2}\sum_{i=1}^n\biggl(\overline{\frac{\partial f}{\partial z_i}}(\mathbf{z})+\frac{\partial f}{\partial \bar{z}_i}(\mathbf{z})\biggr)\bar{v}_i(\mathbf{z}).
\end{split}
\end{equation}
Hence \eqref{it:u} and \eqref{it:w} follow from \eqref{eq:proj.on.ev}.
\end{proof}

\begin{corollary}\label{cor:n+2.frame}
Let $f\colon \C^n \to \C$ be a mixed function with an isolated critical point at the origin. Let $v$ be a continuous  vector field on $V_f$ which is non-zero on $V_f^*$. Let $\mathbf{z}\in V_f^*$. Then, $(\nabla g(\mathbf{z}),v(\mathbf{z}))$ 
is a complex $2$-frame in $\C^n$ if and only if $w(\mathbf{z})\neq0$.
\end{corollary}

\begin{proof}
Since $\mathbf{z}\in V_f^*$ we have that $v(\mathbf{z})\neq0$ and $\nabla g(\mathbf{z})\neq0$. 
By Lemma~\ref{lem:NF.li} $w(\mathbf{z})\neq0$ if and only if the vectors $\{\nabla g(\mathbf{z}),v(\mathbf{z})\}$ are $\C$-linearly independent.
\end{proof}

One particular case of Corollary~\ref{cor:n+2.frame} is when $v$ is  a complex vector field.

\begin{lemma}\label{lem:cvf.PH}
Let $f\colon \C^n \to \C$ be a mixed function with an isolated critical point at the origin. Let $v$ be a continuous complex vector field on $V_f$ which is non-zero on $V_f^*$.
Then, for every $\mathbf{z}\in V_f^*$ we have:
\begin{enumerate}
 \item $u(\mathbf{z})=0$,
 \item The vectors $\nabla g(\mathbf{z})$ and $v(\mathbf{z})$ are orthogonal,
 \item $\nabla g(\mathbf{z})=w(\mathbf{z})\neq0$.
 \end{enumerate}
\end{lemma}

\begin{proof}
Since $v$ is a complex vector field, by Corollary~\ref{cor:complex.vf} the (conjugate of the) first equation of \eqref{eq:mixed} is satisfied. Hence, 
the vectors $\nabla g(\mathbf{z})$ and $v(\mathbf{z})$ are orthogonal and 
by \eqref{eq:u=aev} we have that $u(\mathbf{z})=0$. 
If $\mathbf{z}\in V_f^*$, then $\nabla g(\mathbf{z})\neq0$ and by Lemma~\ref{lem:u}-\eqref{it:w}
we have that $\nabla g(\mathbf{z})=w(\mathbf{z})$.
\end{proof}

\begin{proposition}\label{prop:frame.complex.vf}
Let $f\colon \C^n \to \C$ be a mixed function with an isolated critical point at the origin. Let $v$ be a continuous complex vector field on $V_f$ which is non-zero on $V_f^*$.
Then, for every $\mathbf{z}\in V_f^*$, $(\nabla g(\mathbf{z}), v(\mathbf{z}))$ is a complex $2$-frame in $\C^n$.
\end{proposition}

\begin{proof}
It follows from Lemma~\ref{lem:cvf.PH} and Corollary~\ref{cor:n+2.frame}.
\end{proof}

From Proposition~\ref{prop:frame.complex.vf} we recover the case when $f$ is a holomorphic function.

\begin{corollary}\label{cor:frame.holo}
Let $f\colon \C^n \to \C$ be a holomorphic function with an isolated critical point at the origin. Let $v$ be a continuous vector field on $V_f$ which is non-zero on $V_f^*$.
Then, for every $\mathbf{z}\in V_f^*$, the ordered set of vectors $(\nabla g(\mathbf{z}), v(\mathbf{z}))$ is a complex $2$-frame in $\C^n$.
\end{corollary}

\begin{proof}
For $f$ holomorphic, every vector field on $V_f$ is a complex vector field. The corollary follows from Proposition~\ref{prop:frame.complex.vf}.
\end{proof}

\begin{remark}\label{rem:holo.ng=nf}
If $f\colon \C^n \to \C$ is a holomorphic function we have that $\bar{\mathrm{d}}f(\mathbf{z})=0$, thus by \eqref{eq:nabla.d} we have that 
$\nabla g(\mathbf{z})=\overline{\mathrm{d}(\mathbf{z})}=\overline{\nabla f(\mathbf{z})}$. So Corollary~\ref{cor:frame.holo} recovers the complex $2$-frame $(\overline{\nabla f(\mathbf{z})},v(\mathbf{z}))$. By Lemma~\ref{lemma:CambioW} this $2$-frame is equivalent to the $2$-frame $(v(\mathbf{z}),\overline{\nabla f(\mathbf{z})})$.
\end{remark}

However, there are cases when $w(\mathbf{z})=0$ as shown in the following example.

\begin{example}\label{example:fallaLI}
Let $f \colon \C^3 \to \C$ given by
\begin{equation*}
    f(z_1,z_2,z_3)= -z_1\bar{z}_1^4+z_2^4+z_3^4.
\end{equation*}
The function $f$ is a polar weighted homogeneous polynomial with radial weight type $(4,5,5;20)$ and angular weight type $(-4,3,3;12)$ with an isolated critical point at the origin (see~\cite{Cisneros-Molina:Join,Oka:NDMF}). Consider the vector fields given by the radial and angular action, respectively,
\begin{align*}
	v_{\mathbb{R}^+}(z_1,z_2,z_3) &:= (4z_1,5z_2,5z_3),\\
	v_{\mathbb{S}^1}(z_1,z_2,z_3) &:= (-4iz_1,3iz_2,3iz_3).
\end{align*} 
In this case:
\begin{enumerate}
	\item the vectors  $\{\nabla g(\mathbf{z}), v_{\mathbb{R}^+}(\mathbf{z})\}$ are $\C$-linearly independent but,
	\item the vectors $\{\nabla g(\mathbf{z}), v_{\mathbb{S}^1}(\mathbf{z})\}$ are $\C$-linearly dependent.
\end{enumerate}
We prove the second assertion: thus, we need to find a non-trivial solution to 
\begin{equation}
\label{eq:SitemaFalla1}
\alpha v_{\mathbb{S}^1} (\mathbf{z})+ \beta \nabla g(\mathbf{z}) =0.
\end{equation}
By~\eqref{eq:nablag.z} we get
\begin{equation}
    \nabla g(\mathbf{z}) = \left(-z_1^4-4z_1 \bar{z_1}^3,4\bar{z}_2^3,4\bar{z}_3^3\right).
\end{equation}
Assume that $z_1=x,z_2=z_3=y$ where $x$ and $y$ are real numbers. By~\eqref{eq:SitemaFalla1} we have the following system of equations
\begin{equation*}
\left[
\begin{array}{cc}
-4ix & -5x^4 \\
3iy & 4y^3 \\
\end{array}
\right]
\begin{bmatrix}
\alpha \\
\beta \\
\end{bmatrix} =\begin{bmatrix}
0\\
0\\
\end{bmatrix} .
\end{equation*}
We just need to find real numbers such that the determinant of the previous matrix is zero, i.e., we need to find a real solution of the equation
\begin{equation}
\label{eq:Det.Ejemplo}
iyx\left(-16y^2 + 15 x^3\right)=0.
\end{equation}
It is easy to check that 
\begin{equation*}
x=\frac{128}{225},\quad \quad y=\sqrt{\frac{131072}{759375}},
\end{equation*}
is a solution of equation~\eqref{eq:Det.Ejemplo}. Hence, at the point
\begin{equation}
\label{eq:PuntoFalla}
z_1=\frac{128}{225},\quad \quad z_2=z_3=\sqrt{\frac{131072}{759375}},
\end{equation}
the system of equations \eqref{eq:SitemaFalla1} has a non-trivial solution. Thus, the vectors
\begin{equation*}
\{\nabla g(\mathbf{z}), v_{\mathbb{S}^1}(\mathbf{z})\},
\end{equation*}
are $\C$-linearly dependent, so $w(\mathbf{z})=0$. Furthermore, it is easy to check that the point \eqref{eq:PuntoFalla} is a zero of the function $f$. 
\end{example} 

\subsection{Frame homotopic vector fields}
To avoid situations as in Example~\ref{example:fallaLI}, the strategy is to choose a different election of vector $v^{\perp}(\mathbf{z})$ instead of $\nabla g(\mathbf{z})$, such that  $v^{\perp}(\mathbf{z})$ is always $\C$-linearly independent from the vector $v(\mathbf{z})$. The following definition gives us the idea in our construction.

\begin{definition}
\label{def:BuenaEleccion}
Let $v^{\perp}\colon V_f^* \to \C^n$  be a continuous vector field on $V_f$ which is non-zero on $V_f^*$. Suppose that $v^{\perp}$ and $v$ are complex orthogonal in $\C^n$. We say that $v^{\perp}$ is \emph{frame homotopic to $v$} if there exists a homotopy between the real $3$-frames
\begin{equation*}
(v(\mathbf{z}),\nabla{g}(\mathbf{z}),\nabla{h}(\mathbf{z})) \quad \text{and} \quad  (v^{\perp}(\mathbf{z}),v(\mathbf{z}),iv(\mathbf{z})).
\end{equation*}
\end{definition}
The following lemma guarantee that given any continuous vector field $v$ on $V_f$ which is non-zero on $V_f^*$, there is always a frame homotopic vector field $v^{\perp}$.
\begin{lemma}
\label{lema.generalj}
Let $f\colon \C^n \to \C$ be a mixed function with an isolated critical point at the origin. Let $v$ a continuous vector field on $V_f$ which is non-zero on $V_f^*$.
Then, there exists a continuous vector field $v^{\perp}$ frame homotopic to $v$.
\end{lemma}
\begin{proof}
Since $V_f$ has an isolated singularity at the origin $V_f^*=V_f \setminus \{\mathbf{0}\}$, since $V_f$ is homeomorphic to the cone of $L_f$ (see~\cite[Theorem~2.10]{Milnor:SPCH})
we have that the link $L_f$ of $f$ is a deformation retract of $V_f^*$ (in a small enough representative of the germ of $f$). Hence, in order to prove the lemma it is enough to work with $L_f$.
Up to homotpy, we have the following maps:
\begin{equation*}
    (v,iv)\colon L_f \to V(2,2n), \quad \quad (\nabla{g},\nabla{h})\colon L_f \to V(2,2n),
\end{equation*}
Denote by $[L_f,V(2,2n)]$ the homotopy classes of maps from $L_f$ to $V(2,2n)$.
Since $V(2,2n)$ is $2n-3$ connected and $L_f$ is a smooth compact manifold of dimension $2n-3$, by obstruction theory \cite[page~425 \& Theorem~8.4.3]{Spanier:AlgTop} there is a biyection
\begin{equation*}
[L_f,V(2,2n)]\cong H^{2n-2}(L_f;\Z)=0,
\end{equation*}
thus, there exists a homotopy
\begin{equation*}
    H\colon L_f \times I \to V(2,2n),
\end{equation*}
such that $H(\mathbf{z},0)=(\nabla{g}(\mathbf{z}),\nabla{h}(\mathbf{z}))$ and $H(\mathbf{z},1)=(v(\mathbf{z}),iv(\mathbf{z}))$. By~\eqref{eq:fib.gen}, the projection map
\begin{equation*}
p\colon V(3,2n) \to V(2,2n),
\end{equation*}
is a fibration. Furthermore, the map \eqref{eq:GSV-map.real} given by the real GSV-index,
\begin{equation*}\
\tilde{\varphi}_v= (v,\nabla{g},\nabla {h}) \colon L_f \to V(3,2n),
\end{equation*}
is a lifting of $H(\mathbf{z},0)$. Therefore, by the homotopy lifting property there is a lifting of $H$ to a homotopy
\begin{equation}
\label{eq:HomotopiaBuenaEleccion}
\tilde{H}\colon L_f \times I \to V(3,2n),
\end{equation}
such that $\tilde{H}(\mathbf{z},0)=(v(\mathbf{z}),\nabla{g}(\mathbf{z}),\nabla{h}(\mathbf{z}))$ and $\tilde{H}(\mathbf{z},1)=(v^{\perp}(\mathbf{z}),v(\mathbf{z}),iv(\mathbf{z}))$ for some vector field $v^\perp$
that is never zero and it is orthogonal to $v$ and $iv$ as vectors in $\R^{2n}$. Therefore, $v^{\perp}$ and $v$ are orthogonal as vectors in $\C^n$. Indeed, notice that
\begin{align*}
    \re \inpr[\C^n]{v^{\perp}(\mathbf{z})}{v(\mathbf{z})}&= \inpr[\R^{2n}]{v^{\perp}(\mathbf{z})}{v(\mathbf{z})}=0\\
    \im \inpr[\C^n]{v^{\perp}(\mathbf{z})}{v(\mathbf{z})}&= \inpr[\R^{2n}]{v^{\perp}(\mathbf{z})}{iv(\mathbf{z})}=0.
\end{align*}
Thus, $\inpr[\C^n]{v^{\perp}(\mathbf{z})}{v(\mathbf{z})}=0$. This equality and the homotopy~\eqref{eq:HomotopiaBuenaEleccion} tell us that $v^{\perp}$ is frame homotopic to $v$. This proves the lemma. 
\end{proof}

\subsection{The mixed GSV index} 
Let $f\colon \C^n \to \C$ be a mixed function with an isolated critical point at the origin. As before, denote by $V_f=f^{-1}(0)$.
Let $\mathbb{S}^{2n-1}_\varepsilon\subset\C^n$ be a sphere of  sufficiently small radius $\varepsilon>0$ with center at the origin. Let $L_f=V_f\cap\mathbb{S}^{2n-1}_\varepsilon$ be the link of $f$ at $0$. We have that $L_f$ is a compact oriented smooth manifold of dimension $2n-3$.

Let $v$ be a continuous vector field on $V_f$ with isolated singularity at the origin. Let $v^{\perp}$ be any vector field frame homotopic to $v$. By Definition~\ref{def:BuenaEleccion}, for every point $\mathbf{z}\in V_f^*$ we have that $(v^{\perp}(\mathbf{z}),v(\mathbf{z}))$ is a complex $2$-frame in $\C^{n}$
and up to homotopy, it can be assumed to be orthonormal. Hence we have a continuous map from $V_f^*$ to the Stiefel manifold $W(2,n)$ of complex orthonormal $2$-frames in $\C^{n}$,
which restricted to the link $L_f$ gives a continuous map
\begin{equation}\label{eq:mGSV.map}      
 \phi_v=(v^{\perp}(\mathbf{z}),v(\mathbf{z}))\colon L_f\to W(2,n).
\end{equation}
By \eqref{eq:pii.wkm}, the manifold $W(2,n)$ is $2n-4$-connected and $\pi_{2n-3}\bigl(W(2,n)\bigr)\cong\Z$, thus $\phi_v$ has a degree $\deg(\phi_v)$ defined in \eqref{eq:deg}. First, we prove that the degree of the map $\phi_v$ does not depend of the election of $v^{\perp}$. For this we need a lemma.

Consider the following map between real Stiefel manifolds
\begin{equation}\label{eq:p2}
\begin{gathered}
\tilde{p}_2\colon V(4,2n)\longrightarrow V(3,2n)\\
\tilde{p}_2(v_1,v_2,v_3,v_4)=(v_1,v_3,v_4)
\end{gathered}
\end{equation}
which removes the second vector of the frame. Let $\Phi=\tilde{p}_2\circ q$ be the following composition of the embedding \eqref{eq:comp.real} with the map \eqref{eq:p2}
\begin{equation}\label{eq:Phi}
\Phi\colon W(2,n)\xrightarrow{q}V(4,2n)\xrightarrow{\tilde{p}_2} V(3,2n).
\end{equation}
Hence we have
\begin{equation*}
 \Phi(v_1,v_2)=(v_1,v_2,iv_2).
\end{equation*}

\begin{lemma}\label{lem:fi.fi}
Let $M$ be a oriented closed $(2n-3)$-manifold and consider a map $\phi\colon M\to W(2,n)$ and the composition $\varphi=\Phi\circ\phi\colon M\to V(3,2n)$.
Let $\hat{\phi}\colon M\to\mathbb{S}^{2n-3}$ and $\hat{\varphi}\colon M\to\mathbb{S}^{2n-3}$ be the maps given, respectively, by Propositions~\ref{proposition:mixedGSVobstruction} and \ref{prop:Vkm.smn}.
Then $\hat{\phi}$ and $\hat{\varphi}$ are homotopic.
\end{lemma}

\begin{proof}
Using the definitions of the maps it is straightforward to see that in the following diagram the squares commute and by Proposition~\ref{proposition:mixedGSVobstruction} the triangle commutes up to homotopy, that is,
$\phi\simeq \gamma\circ\hat{\phi}$:
\begin{equation}\label{eq:diag.Stiefel.cr}
\xymatrix{
 & M\ar[d]^{\phi}\ar[ld]_{\hat{\phi}} & \\
\mathbb{S}^{2n-3}\cong W(1,n-1)\ar@{^{(}->}[r]^(.7){\gamma}\ar@{=}[d] & W\bigl(2,n\bigr)\ar[r]^{p}\ar[d]^{\Phi} & W\bigl(1,n\bigr)\ar[d]^{q}\\
\mathbb{S}^{2n-3}\cong V(1,2n-2)\ar@{^{(}->}[r]^(.7){\tilde{\gamma}} & V\bigl(3,2n\bigr)\ar[r]^{\tilde{p}} & V\bigl(2,2n\bigr)
}
\end{equation}
where the first row is fibration \eqref{eq:fib.gen} with $m=n$, $k=2$ and $l=1$; and the second row is fibration \eqref{eq:fib.gen.real}
with $m=2n$, $k=3$ and $l=2$. Hence we have that $\Phi\circ\phi\simeq \Phi\circ\gamma\circ\hat{\phi}=\tilde{\gamma}\circ\hat{\phi}$. By Proposition~\ref{prop:Vkm.smn} we have that
$\Phi\circ\phi\simeq\tilde{\gamma}\circ\hat{\varphi}$ defines $\hat{\varphi}$ and since it is unique up to homotopy we have $\hat{\phi}\simeq\hat{\varphi}$. 
\end{proof}

\begin{lemma}
\label{lema:IndependenciaV}
Let $v_1^{\perp}$ and $v_2^{\perp}$ be two vector fields, both frame homotopic to $v$. Denote by
\begin{align*}
    \phi_{1}\colon L_f &\to W(2,n) \quad \quad &\phi_{2}\colon & L_f \to W(2,n)\\
    \mathbf{z} &\mapsto (v_1^{\perp}(\mathbf{z}),v(\mathbf{z})) \quad \quad & & \mathbf{z} \mapsto (v_2^{\perp}(\mathbf{z}),v(\mathbf{z}))
\end{align*}
Then, we have $\deg(\phi_{1})=\deg(\phi_{2})$.
\end{lemma}

\begin{proof}
By Proposition~\ref{proposition:mixedGSVobstruction}, there exist maps 
\begin{equation*}
    \hat{\phi}_{1}\colon L_f \to \mathbb{S}^{2n-3} \quad \quad \hat{\phi}_{2}\colon L_f \to \mathbb{S}^{2n-3},
\end{equation*}
such that
\begin{equation}
    \label{eq:NoDependeV1}
 \deg(\phi_{1})=\deg(\hat{\phi}_{1}) \quad \text{and} \quad \deg(\phi_{2})=\deg(\hat{\phi}_{2}).
\end{equation}
Consider the compositions with the map $\Phi\colon W(2,n)\to V(3,2n)$
\begin{align*}
    \varphi_{1}=\Phi\circ\phi_1\colon L_f &\to V(3,2n) \quad \quad &\varphi_{2}=\Phi\circ\phi_2\colon & L_f \to V(3,2n)\\
    \mathbf{z} &\mapsto (v_1^{\perp}(\mathbf{z}),v(\mathbf{z}),iv(\mathbf{z})) \quad \quad & & \mathbf{z} \mapsto (v_2^{\perp}(\mathbf{z}),v(\mathbf{z}),iv(\mathbf{z})).
\end{align*}
Since $v_1^{\perp}$ and $v_2^{\perp}$ are frame homotopic we have that $\varphi_{1}\simeq\varphi_{2}$. On the other hand, by Proposition~\ref{prop:Vkm.smn} there exist maps
$\hat{\varphi}_1\colon L_f\to\mathbb{S}^{2n-3}$ and $\hat{\varphi}_2\colon L_f\to\mathbb{S}^{2n-3}$ such that $\tilde{\gamma}\circ\hat{\varphi}_1\simeq \varphi_{1}$ and 
$\tilde{\gamma}\circ\hat{\varphi}_2\simeq \varphi_{2}$. Thus $\tilde{\gamma}\circ\hat{\varphi}_1\simeq \varphi_{1}\simeq\varphi_2$. Since $\hat{\varphi}_2$ is unique up to homotopy we have
$\hat{\varphi}_1\simeq\hat{\varphi}_2$. By Lemma~\ref{lem:fi.fi} we have that $\hat{\phi}_1\simeq\hat{\varphi}_1$ and $\hat{\phi}_2\simeq\hat{\varphi}_2$. Thus $\hat{\phi}_1\simeq\hat{\phi}_2$ and
$\deg(\phi_{1})=\deg(\hat{\phi}_{1})=\deg(\hat{\phi}_{2})=\deg(\phi_{2})$ which proves the lemma.
\end{proof}


\begin{definition}
\label{def:MGSVIndex}
Let $f\colon \C^n \to \C$ be a mixed function with isolated critical point at the origin and let $v$ be a continuous vector field on $V_f$ with isolated singularity at the origin.
Let $v^{\perp}$ be any vector field frame homotopic to $v$ and let $\phi_v=(v^{\perp}(\mathbf{z}),v(\mathbf{z}))\colon L_f\to W(2,n)$ be the map defined in \eqref{eq:mGSV.map}.
The \textit{mixed GSV-index} of $v$ is given by
\begin{equation*}
\Ind[mGSV](v,\mathbf{0})=\deg(\phi_v).
\end{equation*}
\end{definition}
\begin{remark}
By Lemma~\ref{lema:IndependenciaV}, Definition~\ref{def:MGSVIndex} does not depend on the election of $v^{\perp}$.
\end{remark}

\subsection{Comparison with the classical GSV index for holomorphic maps}
In this subsection we prove that if $f\colon\C^n\to\C$ is a holomorphic map with isolated critical point at the origin and $v$ is a vector field on $V_f$ with an isolated zero at the origin, 
then the gradient vector field $\overline{\nabla f}$ is frame homotopic to $v$. Thus, the classical GSV-index and the mixed GSV-index of $v$ coincide, hence, the mixed GSV-index is indeed a generalization of the GSV-index. We prove a stronger result.
\begin{proposition}
\label{prop:goodElectionComplex}
Let $f\colon \C^n \to \C$ be a mixed function with isolated critical point at the origin. Let $v$ be a continuous complex vector field on $V_f$ with isolated singularity at the origin. Then, the vector field $\nabla{g}$ is frame homotopic to $v$.
\end{proposition}
\begin{proof}
Up to homotopy, we can assume that $v(\mathbf{z}), \nabla{g}(\mathbf{z})$ and $\nabla{h}(\mathbf{z})$ are orthogonal over $\R$. The proof is done in four steps.
\paragraph{\textbf{Step one}} Consider the homotopy
\begin{align*}
    L_1 \colon V_f^* \times I &\xrightarrow{\hspace*{2cm}}  V(3,2n)\\
    (\mathbf{z},t) &\mapsto \bigl(v(\mathbf{z}),\nabla{g}(\mathbf{z}), (1-t)\nabla{h}(\mathbf{z})+tiv(\mathbf{z})\bigr).
\end{align*}
By Corollary~\ref{cor:complex.vf} the vectors $\nabla{g}$ and $\nabla{h}$ are orthogonal to $v$ and $iv$. Therefore, for any $(\mathbf{z},t)$ in $V_f^* \times I$ the vector
\begin{equation*}
    (1-t)\nabla{h}(\mathbf{z})+tiv(\mathbf{z}),
\end{equation*}
is never zero and it is orthogonal to $v(\mathbf{z})$ and $\nabla{g}(\mathbf{z})$ . Therefore, the real frames
\begin{equation}
\label{eq:Equiv1}
    \bigl(v(\mathbf{z}), \nabla{g}(\mathbf{z}),\nabla{h}(\mathbf{z})\bigr) \quad \text{and} \quad \bigl(v(\mathbf{z}),\nabla{g}(\mathbf{z}),iv(\mathbf{z})\bigr),
\end{equation}
are equivalent.

\paragraph{\textbf{Step two}} Define the following homotopy
\begin{align*}
    L_2 \colon V_f^* \times I &\xrightarrow{\hspace*{2cm}}  V(3,2n)\\
    (\mathbf{z},t) &\mapsto \bigl(v(\mathbf{z}), (1-t)\nabla{g}(\mathbf{z})+t\nabla{h}(\mathbf{z}),iv(\mathbf{z})\bigr).
\end{align*}
By an analogous argument to \textbf{Case one}, we get that the real frames
\begin{equation}
\label{eq:Equiv2}
    \bigl(v(\mathbf{z}),\nabla{g}(\mathbf{z}),iv(\mathbf{z})\bigr) \quad \text{and} \bigl(v(\mathbf{z}), \nabla{h}(\mathbf{z}),iv(\mathbf{z})\bigr),
\end{equation}
are equivalent.

\paragraph{\textbf{Step three}} Let $L_3$ be the homotopy given by
\begin{align*}
    L_3 \colon V_f^* \times I &\xrightarrow{\hspace*{2cm}}  V(3,2n)\\
    (\mathbf{z},t) &\mapsto \bigl((1-t)v(\mathbf{z})+t\nabla{g}(\mathbf{z}),\nabla{h}(\mathbf{z}), iv(\mathbf{z})\bigr).
\end{align*}
By an analogous argument to \textbf{Case one}, we get that the real frames
\begin{equation}
\label{eq:Equiv2}
    \bigl(v(\mathbf{z}),\nabla{h}(\mathbf{z}), iv(\mathbf{z})\bigr) \quad \text{and} \quad \bigl(\nabla{g}(\mathbf{z}),\nabla{h}(\mathbf{z}), iv(\mathbf{z})\bigr),
\end{equation}
are equivalent.
\paragraph{\textbf{Step four}} Define the homotopy $L_4$ by
\begin{align*}
    L_4 \colon V_f^* \times I &\xrightarrow{\hspace*{2cm}}  V(3,2n)\\
    (\mathbf{z},t) &\mapsto \bigl(\nabla{g}(\mathbf{z}),(1-t)\nabla{h}(\mathbf{z})+tv(\mathbf{z}), iv(\mathbf{z})\bigr).
\end{align*}
By an analogous argument to \textbf{Case one}, we get that the real frames
\begin{equation}
\label{eq:Equiv3}
    \bigl(\nabla{g}(\mathbf{z}),\nabla{h}(\mathbf{z}), iv(\mathbf{z})\bigr) \quad \text{and} \quad \bigl(\nabla{g}(\mathbf{z}),v(\mathbf{z}), iv(\mathbf{z})\bigr).
\end{equation}
are equivalent. By~\eqref{eq:Equiv1},~\eqref{eq:Equiv2} and~\eqref{eq:Equiv3}the real frames
\begin{equation*}
    \bigl(v(\mathbf{z}), \nabla{g}(\mathbf{z}),\nabla{h}(\mathbf{z})\bigr) \quad \text{and} \quad \bigl(\nabla{g}(\mathbf{z}),v(\mathbf{z}), iv(\mathbf{z})\bigr),
\end{equation*}
are equivalent. This proves the proposition.
\end{proof}

\begin{corollary}
\label{cor:MGSVequalGSV}
Let $f\colon \C^n \to \C$ be a holomorphic function with isolated critical point at the origin. Let $v$ be a continuous vector field on $V_f$ with isolated singularity at the origin.
Then the classical GSV-index and the mixed GSV-index of $v$ coincide, that is,
\begin{equation*}
\Ind[mGSV](v,\mathbf{0})=\Ind[GSV](v,\mathbf{0}).
\end{equation*}
\end{corollary}
\begin{proof}
Since $f$ is holomorphic, then $v$ is a complex vector field. By Proposition~\ref{prop:goodElectionComplex} we have that the map $\nabla{g}$ is frame homotopic to $v$. The proof follows by Remark~\ref{rem:holo.ng=nf}
\end{proof}

\subsection{The multi-index}
The classical GSV-index is for vector fields with isolated singularity at the origin on complex analytic varieties. Then, the link of $V$ is necessarily connected, unless $V$ has complex dimension $1$.
This does not happen in the real case, given a mixed function $f\colon\C^n\to\C$ it may happen that $V_f$ has several irreducible components and we need to consider vector field on them with isolated singularity.
Following \cite[Section~2]{Aguilar-Seade-Verjovsky:IVFTIRAS}, this yields to the definition of multi-index.

\begin{definition}
Let $f\colon \C^n \to \C$ be a mixed function with isolated critical point at the origin. Let $v$ be a continuous vector field on $V_f$ with isolated singularity at the origin.
Let $L_1,\dots,L_r$ be the connected components of the link $L_f$. We define the \textit{mixed multi-index} $\Ind[mGSV]^{\mathrm{multi}}(v,\mathbf{0})$ of the vector field $v$ by
\begin{equation*}
\Ind[mGSV]^{\mathrm{multi}}(v,\mathbf{0}):=(\deg(\phi_{v_1}),\dots,\deg(\phi_{v_r})),
\end{equation*}
where $\phi_{v_j}$ is the restriction of $\phi_v$ to the component $L_j$. Notice that one has
\begin{equation*}
\Ind[mGSV](v,\mathbf{0})=\deg(\phi_{v_1})+\dots+\deg(\phi_{v_r}).
\end{equation*}
\end{definition}

\section{Applications}
\label{Applications}
In this section we study the relation between the Mixed GSV-index and some classical invariants like the real GSV-index, the Poincar\'e-Hopf index and the curvatura integra.

\subsection{Mixed GSV-index vs real GSV-index}
In this subsection we prove that the real GSV-index is the reduction modulo two of the mixed GSV-index. 
\begin{theorem}
\label{theorem.soniguales}
Let $f\colon \C^n \to \C$ be a mixed function with isolated critical point at the origin and $v$ a continuous vector field on $V_f$ with isolated singularity at the origin. Then,
\begin{equation*}
\Ind[\R GSV](v,\mathbf{0}) = \Ind[mGSV](v,\mathbf{0}) \mod 2.
\end{equation*}
\end{theorem}
\begin{proof}
Let $L_f$ be the link of $f$ at $0$. Up to homotopy we can assume that ${v(\mathbf{z}),\nabla{g}(\mathbf{z}),\nabla{h}(\mathbf{z})}$ are orthonormal over $\R$. We have two continuous maps (given in~\eqref{eq:mGSV.map} and~\eqref{eq:GSV-map.real}),
\begin{align*}
 \phi_v\colon &L_f \to W(2,n)  &\tilde{\varphi}_v \colon& L_f \to V(3,2n),\\
 &\mathbf{z} \mapsto (v^{\perp}(\mathbf{z}),v(\mathbf{z})) &&\mathbf{z} \mapsto (v(\mathbf{z}),\nabla{g}(\mathbf{z}),\nabla{h}(\mathbf{z}))
\end{align*}
The real, and the mixed GSV-index of $v$ are given by
\begin{equation*}
\Ind[mGSV](v,\mathbf{0})=\deg(\phi_v), \quad \text{and} \quad \Ind[\R GSV](v,\mathbf{0})=\deg_2(\tilde{\varphi}_v).
\end{equation*}
Consider the map defined in \eqref{eq:Phi},
\begin{align*}
    \Phi \colon W(2,n) &\to V(3,2n)\\
    (v_1,v_2) &\mapsto (v_1,v_2,iv_2).
\end{align*}
Thus, we have the map
\begin{align*}
    \Gamma:=\Phi\circ{}\phi_v  \colon L_f &\to V(3,2n)\\
    \mathbf{z} &\mapsto (v^{\perp}(\mathbf{z}),v(\mathbf{z}),iv(\mathbf{z})).
\end{align*}
Since $v^{\perp}$ is frame homotopic to $v$, the real $3$-frames
\begin{equation*}
(v(\mathbf{z}),\nabla{g}(\mathbf{z}),\nabla{h}(\mathbf{z})) \quad \text{and} \quad  (v^{\perp}(\mathbf{z}),v(\mathbf{z}),iv(\mathbf{z})).
\end{equation*}
are homotopic. Therefore, by Proposition~\ref{prop:Vkm.smn} and Remark~\ref{rmk:deg2.v1}
\begin{equation}\label{ed:degreeGamma}
    \deg_2(\Gamma) = \deg(v^{\perp})\mod 2 = \deg(v)\mod 2 =\deg_2(\tilde{\varphi}_v)= \Ind[\R GSV](v,\mathbf{0}).
\end{equation}
In order to finish the proof, we need to determine the degree of $\Gamma$. For this, notice the following diagram
 \begin{equation}
 \label{eq:DiagramadeGamma}
\begin{tikzpicture}
  \matrix (m)[matrix of math nodes,
    nodes in empty cells,text height=2ex, text depth=0.25ex,
    column sep=3.5em,row sep=3em] {
    & L_f\\
    \mathbb{S}^{2n-3}\cong W(1,n-1) & W(2,n) \\
    \mathbb{S}^{2n-3}\cong V(1,2n-2) & V(3,2n) \\};
    \draw[-stealth] (m-1-2) edge node[auto]{$\phi_v$} (m-2-2);
    \draw[-stealth] (m-1-2) edge node[above]{$v^\perp$} (m-2-1);
\draw[-stealth] (m-2-1) edge node[auto]{$\gamma$} (m-2-2);
\draw[-stealth] (m-2-2) edge node[auto]{$\Phi$} (m-3-2);
\draw[-stealth] (m-2-1) edge node[auto]{$=$} (m-3-1);
\draw[-stealth] (m-3-1) edge node[above]{$\tilde{\gamma}$} (m-3-2);
\end{tikzpicture}
\end{equation}
By~\eqref{eq:DiagramadeGamma}, Proposition~\ref{proposition:mixedGSVobstruction} and Proposition~\ref{prop:Vkm.smn},
\begin{equation*}
    \deg(\phi_v) = \deg(v^\perp) \quad \text{and} \quad \deg_2(\Gamma)=\deg_2(\Phi \circ{} \phi_v) = \deg(v^\perp) \mod 2.
\end{equation*}
Therefore,
\begin{equation}
\label{ed:degreeGamma2}
    \deg_2(\Gamma)= \deg(\phi_v) \mod 2 = \Ind[mGSV](v,\mathbf{0}) \mod 2.
\end{equation}
By~\eqref{ed:degreeGamma} and~\eqref{ed:degreeGamma2} we get
\begin{equation*}
   \Ind[\R GSV](v,\mathbf{0})= \Ind[mGSV](v,\mathbf{0}) \mod 2.
\end{equation*}
This proves the theorem.
\end{proof}
\subsection{Mixed GSV-index and the Poincar\'e-Hopf index}
In this section we study the relation with the Poincar\'e-Hopf index of $v$ in a Milnor fiber. For this, let $f\colon \C^n \to \C$ be a mixed function with isolated critical point at the origin.
Let $\mathbb{B}^{2n}_\epsilon$ be a closed ball with center at the origin $\mathbf{0}$ in $\C^n$ of sufficiently small radius $\epsilon>0$.  By the Implicit Function Theorem there exists $\delta$ with 
$0<\delta\ll\epsilon$ such that all the fibers over the open disk $\mathbb{D}_\delta\subset\C$ centered at $0$ of radius $\delta$ are transverse to $\mathbb{S}_\epsilon=\partial\mathbb{B}^{2n}_\epsilon$.
Then, by the Ehresmann Fibration Theorem for manifolds with boundary, the restriction
\begin{equation*}
 f|\colon f^{-1}(\mathbb{D}_\delta\setminus\{0\})\cap\mathbb{B}^{2n}_\epsilon\to \mathbb{D}_\delta\setminus\{0\},
\end{equation*}
is a fiber bundle \cite{Milnor:ISH}, \cite[\S11]{Milnor:SPCH} called the Milnor-L\^e Fibration or fibration on the Milnor tube.
Let $t\in\mathbb{D}_\delta\setminus\{0\}$, the fiber $F_t$ over $t$ is given by
\begin{equation*}
    F_t := f^{-1}(t) \cap \mathbb{B}^{2n}_\epsilon.
\end{equation*}
Since $t$ is a regular value of $f$ near $0$, the Milnor fiber $F_t$ is a $(2n-2)$-manifold with boundary, and its boundary $\partial F_t$ is diffeomorphic to the link $L_f$.

Let $v$ be a continuous vector field on $V_f=f^{-1}(0)$ with isolated singularity at the origin.
By~\cite[Lemma~2.9]{Aguilar-Seade-Verjovsky:IVFTIRAS} we can extend $v$ into a non-singular vector field, also denoted by $v$, defined in a neighborhood of the boundary $\partial F_t$ of $F_t$. By~\cite[Theorem~1.1.2]{Brasselet-etal:VFSV}, we can extend $v$ to a vector field on the whole fiber $F_t$ with only one singular point, say $P$, in the interior of $F_t$. The local Poincar\'e-Hopf index at $P$ is independent of the way how we extend $v$ to the interior of $F_t$, and this number is the Poincar\'e-Hopf index of $v$ in $F_t$ denoted by $\Ind[PH](v,F_t)$. 

Using the previous work we have the following theorem.
\begin{theorem}
\label{theorem.sonigualesPH}
Let $f\colon \C^n \to \C$ be a mixed function with isolated critical point at the origin and $v$ a continuous vector field on $V_f$ with isolated singularity at the origin. Then,
\begin{equation*}
\Ind[mGSV](v,0) = \Ind[PH](v,F_t),
\end{equation*}
where $\Ind[PH](v,F_t)$ is the local Poincar\'e–Hopf index of $v$ on a Milnor fibre $F_t$.
\end{theorem}

\begin{proof}
Let $v^\perp$ be any vector field frame homotopic to $v$. Thus, for any $\mathbf{z}\in V_f^*$,  $(v^\perp(\mathbf{z}), v(\mathbf{z}))$ is a complex $2$-frame in $\C^n$ and up to homotopy it can be assumed to be orthonormal.
Restricting this $2$-frame to the link $L_f$ we get the continuous map given in \eqref{eq:mGSV.map}
\begin{align*}
 \phi_v\colon L_f&\to W(2,n)\\
 \mathbf{z}&\mapsto (v^\perp(\mathbf{z}),v(\mathbf{z})),
\end{align*}
which defines the mixed GSV-index, that is, $\Ind[mGSV](v,0)=\deg(\phi_v)$.
Consider the map
\begin{align*}
 \check{\phi}_v\colon L_f&\to W(2,n)\\
 \mathbf{z}&\mapsto (v(\mathbf{z}),v^\perp(\mathbf{z})).
\end{align*}
By Lemma~\ref{lemma:CambioW} we have 
\begin{equation}\label{eq2:check.phi}
\deg(\check{\phi}_v)=\deg(\phi_v)=\Ind[mGSV](v,0).
\end{equation}
Consider fibration \eqref{eq:fib.gen} with $k=2$ and $m=n$, i.~e., $\mathbb{S}^{2n-3}\xrightarrow{\gamma} W(2,n)\xrightarrow{p} W(1,n)$.
By Proposition~\ref{proposition:mixedGSVobstruction} and Remark~\ref{rmk:deg.v1} the map
\begin{equation*}
v\colon L_f\to \mathbb{S}^{2n-3},
\end{equation*}
is such that $\gamma\circ v=\check{\phi}_v$ and we have that
\begin{equation}\label{eq:deg.v}
\deg(v)=\deg(\check{\phi}_v)=\Ind[mGSV](v,0).
\end{equation}

As we mentioned above, we can extend the vector field $v$ on $L_f$ to a vector field $\hat{v}$ on a Milnor fiber $F_t$, with only one singularity at a point $P$ in the interior of $F_t$.
Let $\mathbb{D}^{2n-2}_\eta$ be a small disk in $F_t$ centered at $P$ of radius $\eta>0$
and $\mathbb{S}^{2n-3}_\eta$ its boundary sphere. Then  $W:=F_t\setminus\mathring{\mathbb{D}}^{2n-2}_{\eta}$, where $\mathring{\mathbb{D}}^{2n-2}_{\eta}$ denotes the interior of $\mathbb{D}^{2n-2}_\eta$,
is an oriented $(2n-2)$-manifold with boundary $\partial W=L_f\sqcup \mathbb{S}^{2n-3}_\eta$, where $\mathbb{S}^{2n-3}_\eta$ has the reverse orientation of the usual one.
Hence, restricting $\hat{v}$ to $\partial W$ we get a map
\begin{equation*}
 \partial\hat{v}\colon \partial W\to W(2,n),
\end{equation*}
whose restriction to $L_f$ is $v$, that is, we have
\begin{equation}\label{eq:p.hat.v}
(\partial\hat{v})|_{L_f}=v.
\end{equation}
Since $\hat{v}$ is defined without zeros on $W$, by \cite[Proposition, p.~110]{Gullemin-Pollack:TopDiff} or \cite[Lemma~1, p.~28]{Milnor:TopDiff} the degree of $\partial\hat{v}$ is zero and we have that
\begin{equation*}
0=\deg(\partial\hat{v})=\deg(\partial\hat{v}|_{L_f})-\deg(\partial\hat{v}|_{\mathbb{S}^{2n-3}_\eta}),
\end{equation*}
thus, by \eqref{eq:p.hat.v} and \eqref{eq:deg.v} we have
\begin{equation*}
\Ind[PH](v,F_t)=\deg(\partial\hat{v}|_{\mathbb{S}^{2n-3}_\eta})=\deg(\partial\hat{v}|_{L_f})=\deg(v)=\Ind[mGSV](v,0).
\end{equation*}
This proves the theorem.
\end{proof}

\subsection{Curvatura integra}
In this subsection we give the relation between the mixed GSV-index and the curvatura integra of a mixed function $f$ which satisfies the strong Milnor condition.
We start recalling the definition of the curvartura integra defined by Kervaire in~\cite{M1956}.
Let $M$ be a $m$-dimensional closed manifold embedded in $\R^n$ with trivial normal bundle. A trivialization of this bundle defines, up to homotopy, a continuous map
\begin{equation*}
    \varphi \colon M \to V(n-m,n).
\end{equation*}
The map $\varphi$ induces an homomorphism in homology
\begin{equation}
\label{eq:CurvaturaIntegra.1}
    \varphi_* \colon H_{m}(M;\mathbb{Z}) \to H_{m}(V(n-m,n);\mathbb{Z}).
\end{equation}
If $m$ is even, the image  $\varphi(M)$ of the fundamental class of $M$ under $\varphi_*$ is an integer, or  an integer modulo $2$  if $m$ is odd, which in either case is called the \emph{curvatura integra} of $M$.

In \cite{Cisneros-Grulha-Seade:OTRAM} Cisneros-Molina, Seade and Grulha defined the curvatura integra of certain real analytic maps, based on Kervaire's curvatura integra.
We recall this definition for the case we are interested. Firstly, we need some definitions. 
\begin{definition}
Let $f\colon \C^n \to \C$ be a mixed function with an isolated critical point. The \emph{canonical pencil} of $f$ is a family $V_\ell$ of real analytic varieties, parametrized by the real projective space $\mathbb{R}\mathbb{P}^{1}$, defined as follows: for each line $\ell$ through the origin in $\C$ we define
\begin{equation*}
    V_\ell := \set{\mathbf{z} \in \C^n}{f(\mathbf{z}) \in \ell}.
\end{equation*}
which is a real analytic variety of dimension $2n-1$.
\end{definition}

\begin{remark}\label{rmk:Ll}
Notice that the union of all $V_\ell$ is the whole ambient space $\C^n$ and their intersection is $f^{-1}(0)$. Moreover, since $f$ has an isolated singularity at the origin, then
each $V_\ell$ is non-singular away from $\mathbf{0}$. Hence, the link of $V_\ell$, which we denote by $L_\ell$, is a smooth submanifold of $\C^n\cong\R^{2n}$ of real dimension $2n-2$, embedded with
trivial normal bundle.
\end{remark}

\begin{definition}
Let $f\colon \C^n \to \C$ be a mixed function with an isolated critical point at $\mathbf{0}$. We say that $f$ has \emph{the strong Milnor property} if for every $\epsilon>0$ (small enough) the map:
\begin{equation*}
    \frac{f}{\|f\|}\colon \mathbb{S}_\epsilon \setminus L_f \to \mathbb{S}^1,
\end{equation*}
is a locally trivial fibration. See \cite[Proposition~3.2]{Cisneros-Seade-Snoussi:d-regular} for a characterization of the strong Milnor property in terms of $d$-regularity.
\end{definition}

\begin{definition}
Let $f\colon \C^n \to \C$ be a mixed function with an isolated critical point, having the strong Milnor property. Let $V_\ell$ be an element in the canonical pencil of $f$. Let $L_\ell$ be the link of $V_\ell$, and by Remark~\ref{rmk:Ll} a trivialization of its normal bundle defines, up to homotopy a map
\begin{equation}\label{eq:phi.Ll}
\varphi\colon L_\ell\to V(2,2n).
\end{equation}
Then, the \emph{curvatura integra} of $f$, denoted by $\Psi(f)$, is defined as:
\begin{equation*}
    \Psi(f):= \varphi(L_\ell)=\varphi_*([L_\ell])
\end{equation*}
where $\varphi(L_\ell)$ is Kervaire's curvatura integra of $L_\ell$ defined as the image of the fundamental class $[L_\ell]$ of $L_\ell$ under the homomorphism $\varphi_*$ induced in homology by \eqref{eq:phi.Ll}.
\end{definition}
The following corollary shows the relation between the curvatura integra and the mixed GSV-index.
\begin{corollary}
\label{cor:CurvaturaIntegra}
Let $f\colon \C^n \to \C$ be a mixed function with an isolated critical point, having the strong Milnor property. Let $v$ be a continuous vector field on $V_f$ such that $v$ is everywhere transverse to the link of $f$. Then,
\begin{equation*}
    \Psi(f)=\Ind[mGSV](v,\mathbf{0}).
\end{equation*}
\end{corollary}
\begin{proof}
Let $F_t$ be a Milnor fiber of $f$. By Theorem~\ref{theorem.sonigualesPH} we have
\begin{equation*}
    \Ind[PH](v,F_t) = \Ind[mGSV](v,\mathbf{0}).
\end{equation*}
By~\cite[Theorem~3.3]{Aguilar-Seade-Verjovsky:IVFTIRAS} and~\cite[Theorem~4]{Cisneros-Grulha-Seade:OTRAM}  we have
\begin{equation*}
    \Ind[PH](v,F_t) = \chi(F_t) \quad \text{and}\quad\Psi(f) = \chi(F_t).
\end{equation*}
This proves the corollary. 
\end{proof}
\bibliography{matart,mypapers,mymatbk,matbook,nuevo}
\bibliographystyle{plain}
\end{document}